\documentclass[10pt, a4paper, oneside]{amsart}

\usepackage[utf8]{inputenc}         
\usepackage[T1]{fontenc}           
\usepackage{amsmath, amssymb, amsthm} 
\usepackage{mathtools}
\usepackage{mathrsfs}
\usepackage{extarrows}
\usepackage{hyperref}
\usepackage[alphabetic, abbrev, msc-links,nobysame]{amsrefs}        
\usepackage{geometry}
\usepackage{tikz-cd}              
\tikzcdset{
arrow style=tikz
}
\usepackage{graphicx}              
\usepackage{newtxtext,newtxmath} 
\usepackage{csquotes} 				
\usepackage{stackrel}               
\usepackage{enumitem}
\usepackage[capitalize,noabbrev]{cleveref}
\usepackage{abstract}

 \newtheoremstyle{thms}
	{}{}{\normalfont}{}{\bfseries }{}{ }
	{\thmname{#1} \thmnumber{#2}.\thmnote{#3}}
	 
\newtheoremstyle{defs}
	{}{}{\normalfont}{}{\bfseries }{}{ }
	{\thmname{#1} \thmnumber{#2}.\thmnote{#3}}

\theoremstyle{thms}
\newtheorem{proposition}{Proposition}[section]
\newtheorem{theorem}[proposition]{Theorem}
\newtheorem{lemma}[proposition]{Lemma}
\newtheorem{corollary}[proposition]{Corollary}

\theoremstyle{defs}
\newtheorem{definition}[proposition]{Definition}
\newtheorem{construction}[proposition]{Construction}
\newtheorem{example}[proposition]{Example}
\newtheorem{remark}[proposition]{Remark}

\newcommand{\ho}{\mathrm{ho}}

\newcommand{\bP}{\mathbb{P}}
\newcommand{\bbU}{\mathbb{U}}

\csundef{Top}
\newcommand{\Top}{\mathrm{Top}}

\newcommand{\bfU}{\mathbf{U}}

\newcommand{\Set}{\mathrm{Set}}                  
\newcommand{\Sp}{\mathrm{Sp}}
\newcommand{\An}{\mathrm{An}}

\newcommand{\uni}{\mathrm{uni}}
\newcommand{\op}{\mathrm{op}}

\newcommand{\incl}{\mathrm{incl}}          
\newcommand{\id}{\mathrm{id}}   
\newcommand{\RO}{\operatorname{RO}}
\newcommand{\CRing}{\mathrm{CRing}}
\newcommand{\pr}{\mathrm{pr}}
\newcommand{\point}{\mathrm{pt}}
\newcommand{\cof}{\mathrm{cof}}

\newcommand{\Mod}{\mathrm{Mod}}

\newcommand{\Th}{\mathrm{Th}}
\newcommand{\res}{\mathrm{res}}
\newcommand{\Or}{\mathrm{Or}}
\newcommand{\Hom}{\operatorname{Hom}}     
\newcommand{\Map}{\operatorname{Map}} 
\newcommand{\map}{\operatorname{map}}              
\newcommand{\colim}{\mathop{\mathrm{colim}}} 
\newcommand{\Sym}{\operatorname{Sym}}
\newcommand{\sm}{\wedge}	
\DeclareMathOperator*{\limone}{lim^{\mathrm{\scriptsize 1}}}

\newcommand{\MU}{\mathrm{MU}}  
\newcommand{\mU}{\mathrm{mU}} 
\newcommand{\MUP}{\mathrm{MUP}}
\newcommand{\mUP}{\mathrm{mUP}}
\newcommand{\MRep}{\mathrm{MRep}}
\newcommand{\MGr}{\mathrm{MGr}}
\newcommand{\M}{{\mathrm{M}}}

\newcommand{\Gr}{\mathrm{Gr}}
\newcommand{\BU}{\mathrm{BU}} 
\newcommand{\Rep}{\mathrm{Rep}}
\newcommand{\bO}{\mathrm{bO}}
\newcommand{\bU}{\mathrm{bU}}
\newcommand{\bUP}{\mathrm{bUP}}
\newcommand{\BUP}{\mathrm{BUP}}

\newcommand{\bC}{\mathbb{C}}

\newcommand{\bN}{\mathbb{N}}

\newcommand{\bR}{\mathbb{R}}
\newcommand{\bS}{\mathbb{S}}

\newcommand{\bZ}{\mathbb{Z}}

\newcommand{\cU}{\mathcal{U}}
\newcommand{\cF}{\mathcal{F}}
\newcommand{\cW}{\mathcal{W}}
\newcommand{\cV}{\mathcal{V}}
\newcommand{\cT}{\mathcal{T}}

\newcommand{\eps}{\epsilon}

\newcommand{\doilink}[1]{\href{http://dx.doi.org/#1}{doi:#1}}

\begin{document}

\title{The Homology of Complex Equivariant Bordism}
\author{Julius Groenjes}
\date{\today}
\maketitle

\begin{abstract}
	Let $A$ be an abelian compact Lie group and let $E$ be an oriented $A$-spectrum.
	We compute the $E$-homology of tom Dieck's homotopical $A$-equivariant complex bordism spectrum $\MU_A$ in two ways, correcting an error in Cole--Greenlees--Kriz (2002). 
	Additionally, we calculate the $E$-homology of the geometric $A$-equivariant complex bordism spectrum $\mU_A$.
\end{abstract}

\section{Introduction}
The bordism spectrum $\MU$ is at the center of many active areas of research.
It is the initial orientable spectrum with the universal formal group law \cite{Q69}.
The study of orientable spectra and their formal group laws has led to proofs of celebrated results, such as the classification of thick subcategories of the category of $p$-local finite spectra \cite{HS98}.
This article is part of an effort to extend these methods to the equivariant world.
For a compact Lie group $G$, the \emph{homotopical bordism $G$-spectrum} $\MU_G$ is a genuine $G$-spectrum introduced as an equivariant cohomology theory by tom Dieck \cite{D70}.
Early advances in the equivariant setting include work by Okonek \cite{O82}, who showed that the cohomology theory $\MU_G^*$ is the initial cohomology theory with Thom classes for $G$-vector bundles.
More recently, Cole, Greenlees, and Kriz \cite{CGK00} discuss equivariant orientations and define equivariant formal group laws for \emph{abelian} compact Lie groups.
For an abelian compact Lie group $A$, Hausmann \cite{H22} proved that the orientation of $\MU_A$ induces an isomorphism from $\pi_*^A(\MU_A)$ to the \emph{equivariant Lazard ring}, the representing ring for equivariant formal group laws.

The article \cite{CGK02} contains an inaccuracy that leads to a mistake in \cite[Thm. 8.2.]{CGK02}, the calculation of the $E$-homology of $\MU_A$ for an orientable $A$-spectrum $E$.
The present article corrects this.
By obtaining two colimit formulae for the $A$-spectrum $\MU_A$, we derive its homology in two ways, in \cref{theorem:intro-A} and \cref{theorem:intro-B}.
We further calculate the homology of the \emph{geometric bordism $A$-spectrum} $\mU_A$, representing bordism of $A$-manifolds with a tangentially stably almost complex structure, which is studied in \cite{Com96}.

\subsection{Orientations of $A$-Spectra}
Let $A$ be an abelian compact Lie group. 
There is a notion of \emph{orientation} for $A$-spectra due to \cite{C96}.
We let $\cU_A$ be a (complex) complete $A$-universe, a unitary $A$-representation in which any other countably infinite-dimensional $A$-representation may be embedded.
For $E$, a commutative $A$-ring spectrum, an \emph{orientation} of $E$ is a class $x(\eps)\in E^2_A(\bC P (\cU_A),\bC P(\eps))$, where $\eps$ is the trivial representation on $\bC$, such that:
\begin{enumerate}[label = (\roman*)]
	\item
		For any character $\alpha$ of $A$, we view $\alpha$ and $\alpha^{-1}$ as one-dimensional $A$-representations, and pulling back $x(\eps)$ along an embedding  $i \colon \bC P(\alpha\oplus \eps)\to \bC P(\cU_A)$ yields an $\RO(A)$-graded unit in $E^2_A(S^{\alpha^{-1}})$: 
		\[
			i^*\colon E^2_A(\bC P(\cU_A),\bC P(\eps))\to E^2_A(\bC P(\alpha\oplus \eps), \bC P(\eps))\cong E^2_A(S^{\alpha^{-1}}).
		\]
	\item
		For $\alpha = \eps$, the class $x(\eps)$ restricts to $\Sigma^2 1$ in $E^2_A(S^{\eps^{-1}}) = E^2_A(S^2)$.
\end{enumerate}

\subsection{Calculation of Homology}
Let $E$ be an oriented $A$-spectrum.
Taking Thom spectra of the tautological bundle $\gamma_1$ over $\bC P(\cU_A)$ or the associated negative (virtual) bundle gives rise to genuine $A$-spectra $\M(1)_A$ and $\M(-1)_A$.
Let us write $R_E\coloneqq E_*^A(\bC P(\cU_A)_+)$, as well as $R_E^+\coloneqq E_*^A(\M(1)_A)$, and $R_E^-\coloneqq E_*^A(\M(-1)_A)$.
The Thom isomorphism yields equivalences
\[
	\Th^+\colon R_E[-2]\to R_E^+, \quad \Th^-\colon R_E[2]\to R_E^-,
\]
where $[-]$ denotes the shift of a $\bZ$-graded $E_*^A$-algebras.
Passing from lines to general $n$-planes, consider the total Grassmannian $\Gr(\cU_A)$ of finite-dimensional planes in $\cU_A$, which comes with a tautological bundle $\gamma$ (of non-constant rank).
This bundle and its (virtual) negative yield the Thom spectra $\MRep_A$ and $\MGr_A$, respectively.
In \cref{section:4} and \cref{section:5} we define isomorphisms
\[
	B\colon \Sym(R_E^+)\xrightarrow{\cong}E_*^A(\MRep_A) , \quad C\colon  \Sym(R_E^-) \xrightarrow{\cong} E_*^A(\MGr_A).
\]
We then connect the spectra above to the bordism spectra $\MU_A$ and $\mU_A$.
There are periodic versions of these, given by $\MUP_A$ and $\mUP_A$, which admit ring structures such that the inclusions $\MU_A\to \MUP_A$ and $\mU_A\to \mUP_A$ are inclusions of subrings.
On underlying $A$-spectra, these are summand inclusions.
For an overview of the various Thom spectra and their relations, see \cref{remark:thom-spectra-map}.
A character $\alpha\in A^*$ of $A$ defines a one-dimensional representation.
An inclusion of representations $\alpha\to \cU_A$ induces a map $E_*^A\cong E_*^A(\bC P(\alpha)_+)\to E_*^A(\bC P(\cU_A)_+)$.
We refer to the image of 1 under this map as the \emph{$\alpha$-coaugmentation class} $\vartheta_\alpha$.
We further write $u_\alpha^+ \coloneqq \Th^+(\vartheta_\alpha)$ and $u_\alpha^- \coloneqq \Th^-(\vartheta_\alpha)$.
The first main result states:
\begin{theorem}\label{theorem:intro-A}
	There is an $E_*^A$-algebra isomorphism
	\[
		E_*^A(\MUP_A)\cong \Sym(R_E^+)[(u_\alpha^+)^{-1}\mid \alpha\in A^*].
	\]
\end{theorem}
The above theorem, proven as \cref{theorem:A}, makes use of an equivalence
\[
	\colim_{V\leq \cU_A} \pi_*^A(\Omega^V \MRep_A) \xrightarrow{\simeq} \pi_*^A(\MUP_A),
\]
where the colimit is taken over the poset of finite-dimensional subrepresentations of $\cU_G$ along transition maps multiplying by \emph{unshifted Thom classes}.
We use the fact that $E$-homology commutes with this colimit to find $E_*^A(\MUP_A)$.
The mistake in \cite{CGK02} occurs in the analysis of the colimit diagram above, after applying $E_*^A(-)$.
Contrary to \cite[7.3]{CGK02}, for representations $V\leq W\leq \cU_A$, the transition map $E_*^A(\Omega^V\MRep_A)\to E_*^A(\Omega^W\MRep_A)$ depends on the orthogonal complement $W-V$ as a representation, not only on its dimension.
As a result, for $\alpha$ an irreducible representation, $u_\alpha^+$ is invertible in the homology ring.
Using $C$ and the natural maps $\MGr_A\to \mUP_A$ and $\MGr_A\to \MUP_A$, we show:
\begin{theorem}\label{theorem:intro-B}
	There are $E_*^A$-algebra isomorphisms
	\begin{align*}
		E_*^A(\MUP_A)&\cong \Sym(R_E^-)[(u_\alpha^-)^{-1}\mid \alpha\in A^*],\\
		E_*^A(\mUP_A)&\cong \Sym(R_E^-)[(u_\eps^-)^{-1}].
	\end{align*}
\end{theorem}
For the proof of the above in \cref{theorem:B} we again rely on colimit formulae:
\[
	\colim_{n\in \bN} \pi_*^A(\Omega^{\eps^n}\MGr_A) \xrightarrow{\simeq} \pi_*^A(\mUP_A), \quad \colim_{V\subseteq \cU_A} \pi_*^A(\Omega^{V}\MGr_A) \xrightarrow{\simeq} \pi_*^A(\MUP_A).
\]

\subsection{Coordinatized Statements}
We can give a more explicit description of the above results.
Let $\cF = \{V_i\}_i$ with $|V_i| = i$ be a flag of $\cU_A$, an exhaustive sequence of increasing invariant subspaces, starting with $V_1 = \eps$, which we fix for the rest of this section.
As shown in \cite{C96}, for an oriented $A$-spectrum $E$, there is an $E_*^A$-module isomorphism
\[
 	R_E=E_*^A(\bC P(\cU_A)_+)\cong \bigoplus_{i=0}^\infty E_*^A\{\beta_i^\cF\},
\]
with $\beta_0^\cF = \vartheta_\eps$.
Let us write $b_i^+\coloneqq \Th^+(\beta^\cF_i)$ and $b_i^-\coloneqq \Th^-(\beta^\cF_i)$, leaving the flag implicit.
By using the map $E_*^A(\MRep_A)\to E_*^A(\MUP_A)$, we specialize \cref{theorem:intro-A} to
\[
	E_*^A(\MUP_A) \cong E_*^A[b_0^+, b_1^+,b_2^+,\cdots ][(u^+_\alpha)^{-1}\mid \alpha\in A^*],.
\]
The summand inclusion $\MU_A\to \MUP_A$ defines a sub-algebra in homology.
\begin{corollary}\label{corollary:intro-A-coordinatized}
	Writing $\widetilde x$ for the product $(u^+_\eps)^{-1}\cdot x$, the above restricts to an $E_*^A$-algebra isomorphism
	\[
		E_*^A(\MU_A)\cong E_*^A[\widetilde {b_1^+}, \widetilde {b_2^+},\widetilde {b_3^+}, \cdots ][(\widetilde {u_\alpha^+})^{-1}\mid \alpha\in A^*].
	\]
\end{corollary}
The above is proven in \cref{corollary:A-coordinatized}.
The notation is consistent with \cite{CGK02}, and as mentioned, the difference to their result is the inversion of the elements $b(\vartheta_\alpha)$.
We can also give coordinatized versions of the results obtained from the negative Thom spectrum $\MGr_A$.
Firstly, we find that 
\begin{align*}
	E_*^A(\MUP_A) &\cong E_*^A[b_0^-, b_1^-,b_2^-,\cdots ][(u^-_\alpha)^{-1}\mid \alpha\in A^*],\\
	E_*^A(\mUP_A) &\cong E_*^A[b_0^-, b_1^-,b_2^-,\cdots ][(u^-_\eps)^{-1}].
\end{align*}
We move on to prove the following in \cref{corollary:B-coordinatized}.
\begin{corollary}\label{corollary:intro-B-coordinatized}
	Writing $\widehat x$ for the product $(u^-_\eps)^{-1}\cdot x$, the above restricts to $E_*^A$-algebra isomorphisms
	\begin{align*}
	E_*^A(\MU_A)&\cong E_*^A[\widehat{b_1^-}, \widehat{b_2^-}, \widehat{b_3^-}, \cdots ][ (\widehat{u_\alpha^-})^{-1}\mid \alpha\in A^*],\\
		E_*^A(\mU_A)&\cong E_*^A[\widehat{b_1^-}, \widehat{b_2^-}, \widehat{b_3^-}, \cdots ].
	\end{align*}
		
\end{corollary}

\subsection{Universality of $\MU_A$}
We give a construction of an orientation for the $A$-spectrum $\MU_A$, the \emph{universal orientation} $x^{\uni}(\eps) \in \MU_A^2(\bC P(\cU_A),\bC P(\eps))$.
Orientations can be pushed forward along a map of commutative (homotopy) $A$-ring spectra.
In \cref{section:6} we recover \cite[Thm. 1.2.]{CGK02}:
\begin{theorem}\label{theorem:intro-universality}
	Let $E$ be a commutative (homotopy) $A$-ring spectrum.
	There is a natural bijection
	\[
		[\MU_A,E]_{\CRing_A}\to \Or(E), \quad f\mapsto f_*(x^\uni(\eps))
	\]
between orientations of $E$ and homotopy classes of ring maps $\MU_A\to E$.
\end{theorem}
However, due to the new calculation of $E_*^A(\MU_A)$, the proof given in \cref{theorem:classifies-orientations} requires an additional argument.
Writing $\M(n)_A$ for the Thom $A$-spectrum associated to the tautological bundle $\gamma_n$ over $\Gr_n(\cU_G)$, there is a colimit
\[
	\colim_{V\leq \cU_A} \pi_*^A(\Omega^V\M(n)_A)\xrightarrow{\simeq} \pi_*^A(\MU_A).
\]
Passing to $E_A^*$-cohomology, the following can be thought of as a $\limone$-term vanishing result, \cref{proposition:lim1-vanishing}.
\begin{proposition}\label{proposition:intro-lim1}
	Let $E$ be an oriented $A$-spectrum.
	There is an equivalence of $E^*_A$-modules
	\[
		\lim_{(V\leq \cU_A)^{\op}} {E^*_A}(\Omega^V\M(n)_A)\xrightarrow{\simeq} {E^*_A}(\MU_A),
	\]
\end{proposition}
One can check that all subsequent results of \cite{CGK02} remain true.

\subsection*{Outline of Document}
In \cref{section:1} we define orientations of $A$-spectra and give names to many important (co)homology classes existing for oriented spectra.
Then, in \cref{section:2} we recall the calculation of the homology of Grassmannians of complete universes.
It is in \cref{section:3} that we introduce the relevant bordism spectra as the genuine $A$-spectra underlying certain unitary $A$-prespectra.
Moving on, in \cref{section:4} and \cref{section:5} we calculate the homology of the bordism spectra $\MU_A$ and $\mU_A$, proving theorems A and B.
Lastly, in \cref{section:6}, we comment on the cohomology of $\MU_A$ and prove that it is the universal $A$-spectrum with an orientation.

\subsection*{Acknowledgements}
This work emerged from the master's thesis of the author.
He would like to thank Markus Hausmann for support while and after writing the thesis and John Greenlees for an encouraging and helpful conversation.
Further, he thanks Steffen Sagave for comments on an earlier draft, and his friends for continuous encouragement.
The author would like to thank the Isaac Newton Institute for Mathematical Sciences, Cambridge, for support and hospitality during the program Equivariant homotopy theory in context where work on this paper was undertaken.
This work was supported by EPSRC grant no. EP/Z000580/1 and the Dutch Research Council (NWO) grant number OCENW.M.22.358.

\section{Equivariant orientations}\label{section:1}

We lay out the setting of this article, first introducing various 1-categories and $\infty$-categories from equivariant stable and unstable homotopy theory.
Further, we recall the vocabulary of orientations for equivariant ring spectra.

\subsection{Equivariant Homotopy Theory}

The letter $A$ always denotes an abelian compact Lie group, and $G$ always denotes a (general) compact Lie group.
For all results relating to orientations, we will be restricting to $A$.
However, when this is not needed, we work in the $G$-equivariant context.
\begin{definition}\label{definition:G-spaces} 
	Let $\Top_G$ denote the \emph{category of topological $G$-spaces}, the (1-)category of compactly generated weak Hausdorff spaces with a continuous $G$-action.
	A map $f\colon X \to Y$ is a \emph{$G$-weak equivalence} if for all closed subgroups $H$ of $G$, the maps $f^H\colon X^H\to Y^H$ on fixed points are weak homotopy equivalences.

	The $\infty$-category $\An_G$ of \emph{$G$-animae} is defined as the Dwyer-Kan localization of the nerve of $N(\Top_G)$ at the $G$-weak equivalences.
	We denote the $\infty$-category of pointed $G$-animae by $\An_{G,*}$.
	Note that the Dwyer-Kan localization of the nerve $N(\Top_{G,*})$ at the underlying $G$-weak equivalences is equivalent to $l\colon N(\Top_{G,*})\to \An_{G,*}$.
\end{definition}

Besides $G$-animae, we also work with equivariant objects in the stable setting.

\begin{definition}
	Let $\Sp_G$ denote the \emph{$\infty$-category of genuine $G$-spectra} as in \cite[Cor. C.7.]{GM23}.
	An object $X\in \Sp_G$ is called a \emph{$G$-spectrum}.
	In this reference, Gepner and Meier equip $\Sp_G$ with a symmetric monoidal structure with product $\otimes$, referred to as the \emph{smash product} of genuine $G$-spectra. 
	Its unit is the \emph{$G$-sphere spectrum} $\bS$.
	The \emph{$G$-homotopy groups} of a $G$-spectrum $X$ are defined by 
	\[
	\pi^G_k(X)\coloneqq \pi_0(\map_{\Sp_G}(\bS^k, X)).
	\]
\end{definition}

We further make use of the fact that for a closed subgroup $H\leq G$, there is a monoidal restriction functor $\res_G^H\colon \Sp_G\to \Sp_H$,
which follows from \cite[Cor. C.7]{GM23}.
To define an orientation for an equivariant spectrum, we require a weak form of ring structure on it.
\begin{definition}
	A \emph{$G$-ring spectrum} is a $G$-spectrum $E$ together with 
	\begin{enumerate}[label = (\roman*)]
		\item 
			A \emph{unit map} $1\colon \bS\to E$,
		\item 
			A \emph{multiplication map} $\mu\colon E\otimes E\to E$,
	\end{enumerate}
	such that left- and right-unitality as well as associativity hold in the homotopy category, requiring no higher coherences.
	A $G$-ring spectrum $E$ is \emph{commutative} if the multiplication is commutative in the homotopy category.
	We abuse notation and simply call a $G$-spectrum $E$ a $G$-ring spectrum, suppressing unit and multiplication from the notation.
	A \emph{morphism of $G$-ring spectra} from $R$ to $S$ is a map of $G$-spectra $f\colon R\to S$ such that the requisite diagrams commute in the homotopy category.
	\end{definition}

\begin{definition}
	Let $R$ be a commutative $G$-ring spectrum. An \emph{$R$-module} consists of a $G$-spectrum $M$ and an \emph{action map} $a\colon R\otimes M\to M$ such that unitality and associativity hold in the homotopy category.
	Again, we often refer to $M$ itself as the module.
	A \emph{morphism of $R$-modules} from $M$ to $N$ is a map of $G$-spectra $g\colon M\to N$ such that the requisite diagrams commute in the homotopy category.
\end{definition}

\begin{definition}
	Let $E$ be a $G$-spectrum.
	We define the \emph{$E$-(co)homology} of a $G$-spectrum $X$ as the $\bZ$-graded abelian group
	\[
		E_*^G(X) \coloneqq \pi_*^G(E\otimes X), \quad E^*_G(X) \coloneqq \pi_{-*}^G(\Map_{\Sp_G}(X, E)),
	\]
	for $\Map_{\Sp_G}$ the internal mapping $G$-spectrum functor.
	In the case of a suspension spectrum, we simply write $E_*^G(Y)\coloneqq E_*^G(\Sigma^\infty Y)$ for a pointed $G$-anima $Y$.
	Lastly, if $Z$ is a pointed topological $G$-space, we abbreviate $E_*^G(l(Z))$ to $E_*^G(Z)$.
	The same conventions apply to cohomology.
\end{definition}

\begin{definition}
	Let $E$ be a commutative $G$-ring spectrum and let $V$ and $W$ be unitary $G$-representations.
	We write $\bS^V\coloneqq \Sigma^\infty l(S^V)$ for the associated suspension spectrum.
	A map $\varphi\colon \bS^V\to E\otimes \bS^W$ is an \emph{$\RO(G)$-graded unit} if there exists a map $\psi\colon \bS^{W}\to E\otimes \bS^V$ and an equivalence from $1\otimes \bS^{V\oplus W}$ to the composite
	\[
		\bS^{V\oplus W}\xrightarrow{\varphi\otimes \psi} E\otimes \bS^W\otimes E\otimes \bS^V \simeq E\otimes E\otimes \bS^{V\oplus W} \xrightarrow{\mu\otimes \id} E\otimes \bS^{V\oplus W}
	\]
	where we first shuffle the spheres and then multiply.
\end{definition}

\subsection{Orientations of Equivariant Spectra}
	
In Cole's thesis \cite{C96}, he defines	orientations for $A$-spectra, where $A$ is an abelian compact Lie group.
This notion was subsequently studied in works including \cite{CGK00} and \cite{CGK02} by Cole, Greenlees, and Kriz.		
In the non-equivariant setting, an orientation for a spectrum $E$ is given by a class in $E^2(\bC P^\infty)$, which restricts to $\Sigma^2 1$ in $E^2(\bC P^1)\cong E^2(S^2)$.
In the equivariant setting, the situation is similar: an orientation is a cohomology class over the space which classifies (complex) \emph{equivariant} line bundles.
As in the non-equivariant case, this is modeled by a complex projective space: 
\[
	\BU(1)_G \simeq \bC P(\cU_G),
\]
where $\cU_G$ is a complete (unitary) $G$-universe.
Generally, we take all representations to be unitary representations unless explicitly mentioned otherwise.

\begin{definition}
	A \emph{complete (unitary) $G$-universe} $\cU_G$ is a countably infinite-dimensional unitary $G$-representation which admits an equivariant embedding from any finite $G$-representation.
\end{definition}

\begin{definition}
	Let $\cV$ be an at most countably infinite-dimensional $G$-representation and $n\geq 0$. 
	The \emph{Grassmannian of $n$-planes in $\cV$} is the $G$-space $\Gr_n(\cV)$ of $n$-planes in $\cV$ with $G$-action through $\cV$.
	We write $\bC P(\cV)$ for $\Gr_1(\cV)$. 
	The \emph{tautological bundle} over $\Gr_n(\cV)$ is the $G$-vector bundle
	\[
		\gamma_n = \{(x,L)\in \cV\times \Gr_n(\cV)\mid x\in L\} \xrightarrow{\pr} \Gr_n(\cV).
	\]
	The Thom space of this bundle is denoted by $ \Gr_n(\cV)^{\gamma_n}$.
\end{definition}

In the case of a complete $G$-universe, the bundle $\gamma_n$ over $ \Gr_n(\cU_G)$ is the universal $G$-vector bundle. 
For the following result, see e.g. \cite[Prop. 1.1.30.]{S18}.

\begin{proposition}\label{lemma:classifies-bundles}
	Let $\cU_G$ be a complete $G$-universe.
	The space $\Gr_n(\cU_G)$ with its tautological $G$-vector bundle $\gamma_n$ classifies $n$-dimensional $G$-vector bundles over paracompact $G$-spaces.
	This means that they are $G$-classifying spaces $\BU(n)_G\simeq \Gr_n(\cU_G)$.
\end{proposition}

We now restrict to the $A$-equivariant setting.
The key property of abelian compact Lie groups which we use is the following consequence of Schur's lemma:
\begin{proposition}
	Let $A$ be a compact abelian Lie group. Then all irreducible (complex) $A$-representations are one-dimensional.
\end{proposition}

\begin{definition}
	Let $A^*\coloneqq \Hom(A,U(1))$ denote the group of characters of $A$.
	We view $\alpha\in A^*$ as one-dimensional representations.	
	Let $V$ be an $A$-representation.
	Precomposing the action map by the inversion of $A$ yields another representation on the inner product space $V$, which we will denote by $V^{-1}$.
\end{definition}

\begin{remark}\label{lemma:equivariant-maps-contractible}
	As proven in \cite[II, Lem. 1.5]{LMS86}, for an at most countably infinite-dimensional $G$-representation $\cW$, the space of equivariant embeddings $\cW \to \cU_G$ is equivariantly contractible.
	For this reason, the homotopy class of $\bC P(j)\colon \bC P(\cW)\to \bC P(\cU_G)$ is independent of the choice of equivariant  embedding $j\colon \cW\to \cU_G$.
\end{remark}

Let $\alpha$ be a character and let $\eps$ be the trivial representation on the same vector space.
By $S^{\alpha^{-1}}$ we denote the one-point-compactification of the inverse representation $\alpha^{-1}$ defined above.
The following assignment then defines an equivariant equivalence from this sphere to a projective space:
\begin{align*}\label{equation:q-map}
	q_\alpha \colon S^{\alpha^{-1}}\to \bC P(\alpha\oplus \eps), \quad x\in \alpha^{-1}\mapsto [1:x], \quad \infty\mapsto [0:1]. 
\end{align*}

We now come to the central notion of an orientation.

\begin{definition}\label{definition:orientable-spectrum}
	Let $E$ be a commutative $A$-ring spectrum.
	An \emph{orientation class} for $E$ is a cohomology class $x(\eps)\in E^2_A(\bC P(\cU_A),\bC P(\eps))$ such that:
	\begin{enumerate}[label = (\roman*)]
		\item
			For any one-dimensional representation $\alpha$, restricting along 
			\[
				S^{\alpha^{-1}} \xrightarrow{q_\alpha} \bC P(\alpha\oplus \eps)\xrightarrow{\incl} \bC P(\cU_A)
			\]
			sends the class $x(\eps)$ to an $\RO(A)$-graded unit in $E_A^*(S^{\alpha^{-1}})$.
		\item
			For the trivial representation $\eps$, the class $x(\eps)$ restricts to $\Sigma^2 1$ in $E^2(S^{\eps^{-1}}) = E^2(S^2)$.
	\end{enumerate}
	An \emph{oriented spectrum} $(E,x(\eps))$ consists of a commutative $A$-ring spectrum and an orientation class.
\end{definition}

We now define many related classes in the cohomology of an oriented $A$-spectrum $E$, which can be obtained from the orientation class.
Given a character $\alpha$, there exists an equivariant isometric embedding $j\colon \cU_A\otimes\alpha\to \cU_A$, leading to a map
\[
	m_{\alpha} \colon \bC P(\cU_A)\xrightarrow{-\otimes \alpha} \bC P(\cU_A\otimes \alpha) \xrightarrow{\bC P(j)}  \bC P(\cU_A), \quad W\mapsto   j(W\otimes \alpha).
\]
By \cref{lemma:equivariant-maps-contractible}, the homotopy class of $m_\alpha$ does not depend on the chosen embedding $j$.
\begin{definition}
	Forgetting the basepoint, $x(\eps)$ defines a class $y(\eps)\in E^2_A(\bC P(\cU_A)_+)$ which we call the \emph{universal Euler class}.
	For an irreducible representation $\alpha$, we define  $l_\alpha \coloneqq  (m_{\alpha^{-1},+})^*$.
	The \emph{universal $\alpha$-Euler class} is given by $y(\alpha) \coloneqq l_{\alpha}(y(\eps))\in  E^2_A(\bC P(\cU_A)_+)$.
\end{definition}
Note that $y(\alpha)$ vanishes on $\bC P(\alpha)$.
We will now explore how the notion of an orientation also gives rise to a \emph{Thom-class}.

\begin{construction}
	Let $\cU_A$ be a complete universe.
	The Thom space $\bC P(\cU_A)^{\gamma_1}$ is defined as the quotient of the disk bundle by the sphere bundle.
	We introduce the pair of equivariant maps
	\begin{align*}
		s_0\colon \bC P(\cU_A)\to \bC P(\cU_A)^{\gamma_1},& \quad [l]\mapsto ([l],0),\\
		\rho\colon \bC P(\cU_A)^{\gamma_1}\to \bC P(\eps\oplus \cU_A),& \quad ([y], x\in [y]) \mapsto [\langle x,y \rangle \cdot 1_\eps + (1-|x|)y].
	\end{align*}
	One can conclude that these maps are equivariant homotopy equivalences as in the non-equivariant setting by making use of the fact that the inclusion $\rho\circ s_0 = \bC P(0\oplus \id)\colon \bC P(\cU_A)\to \bC P(\eps\oplus \cU_A)$ is a homotopy equivalence.
\end{construction}

\begin{definition}\label{definition:thom-class}
	The \emph{Thom class} $t \in E^2_A(\bC P(\cU_A)^{\gamma_1})$ is given by $\rho^*(y(\eps))$.
	Equivalently, $t$ is determined by fixing its Euler class $s_0^*(t)$ to be $y(\eps)$.
\end{definition}

\begin{definition}\label{definition:y-alphas}
	Let $V$ be an $A$-representation with a decomposition into irreducibles  $V\cong \alpha_1\oplus \cdots \oplus \alpha_n$ and denote the induced equivalence on spheres by $j\colon S^{V}\to \bigwedge_{i}S^{\alpha_i}$.
	\begin{enumerate}[label = (\roman*)]
		\item
			The \emph{universal $V$-Euler class} is given by $y(V) = y(\alpha_1)\cdot y(\alpha_2)\cdot \ldots \cdot y(\alpha_n)$.
		\item
			For an irreducible representation $\alpha$, we obtain a homotopy class $c(\alpha)\colon \bC P(\alpha)\to \bC P(\cU_A)$.
			The map on Thom spaces is given by $\Th(c(\alpha))\colon S^\alpha\to \bC P(\cU_A)^{\gamma_1}$. 
			The \emph{Thom class} of $\alpha$ is given by $\chi(\alpha) = (\Th(c(\alpha))^*(t)\in E^2_A(S^\alpha)$.
		\item
			The \emph{Thom class} of $V$ is defined to be
			\[
				\chi(V)\coloneqq j^*(\chi(\alpha_1)\chi(\alpha_2) \cdots \chi(\alpha_n))  \in E_A^{2|V|}(S^{V}).
			\] 
		\item
			The \emph{Euler class} of $V$ is $e(V)\coloneqq i^*(\chi(V))\in E_A^{|V|}(S^0)$, the pullback of $\chi(V)$ along the inclusion $i\colon S^0\to S^{V}$. In the case that the fixed points of $V$ are nontrivial, the Euler class vanishes.
	\end{enumerate}
\end{definition}

\begin{remark}
	The classes $y(V)$, $\chi(V)$, and $e(V)$ are independent of the chosen decomposition of $V$, as the space of decompositions is connected, so any two equivalences $j$ or $j'$ are homotopic.
	By \cref{theorem:orientation-yields-thom-classes}, one can define a higher Thom class for bundles such as $\gamma_1\otimes V$, and it can be deduced that $y(V)$ is the Euler class of this Thom class by the multiplicativity and naturality of Thom classes regarding the $l_{\alpha_i}$.
\end{remark}

\begin{lemma}\label{lemma:thom-class-two-ways}
	Let $E$ be an orientable spectrum and let $\alpha$ be an irreducible representation.
	The Thom class $\chi(\alpha)$ agrees with the $\RO(A)$-graded units in the definition of an orientation:
	\[
		\chi(\alpha) = \incl^* (q_{\alpha^{-1}})^*(x(\eps))\in E_A^*(S^\alpha).
	\]
	Further, all $\chi(V)\in E^{|V|}_A(S^V)$ are $\RO(A)$-graded units.
\end{lemma}
\begin{proof}
	For the first claim, we use the commutative diagram below and the definition of the Thom class $\rho^*(y(\eps)) = t$.
	\[
		\begin{tikzcd}
		S^\alpha \rar["\Th(c(\alpha))"]	\dar["q_{\alpha^{-1}}"]& \bC P(\cU_A)^{\gamma_1} \dar["\rho"] \\
		\bC P(\alpha^{-1}\oplus \eps) \rar["\incl"] & \bC P(\eps\oplus \cU_A)
		\end{tikzcd}
	\]
	Firstly, unit multiply to units, and secondly, the restriction along the bottom path is an $\RO(A)$-graded unit by virtue of $x(\eps)$ being an orientation class.
	Hence, the second claim follows from the first.
\end{proof}

\begin{lemma}\label{lemma:thom-class-is-orientation}
	Let $E$ be a commutative $A$-ring spectrum.
	Let $t\in E^2_A(\bC P(\cU_A)^{\gamma_1})$ be a class such that for every irreducible representation $\alpha$, the restriction $\Th(c(\alpha))^*(t)\in E^2_A(S^{\alpha})$ is an $\RO(A)$-graded unit and $\Th(c(\eps))^*(t) = E^2_A(S^2)$ is the unit of $E$.
	Then $s_0^*(t)\in E_A^2(\bC P(\cU_A)_+)$ vanishes over $\bC P(\eps)$ and defines an orientation of $E$.
\end{lemma}
\begin{proof}
	By the above, the only claim that remains is the vanishing of $s_0^*(t)$ over $\bC P(\eps)$.
	This in turn follows from the fact that the point $s_0(\bC P(\eps))$ is connected to the basepoint in the fixed points $(\bC P(\cU_A)^{\gamma_1})^A$.
\end{proof}

As morphisms of $A$-ring spectra preserve the unit and the multiplication, we obtain the following consequences:
\begin{lemma}\label{lemma:pushforward-orientation}
	Let $(E,x(\eps))$ be an oriented spectrum, $R$ a commutative $A$-ring spectrum, and $B\leq A$ a closed subgroup of $A$.
	\begin{enumerate}[label = (\roman*)]
		\item 
			If $f\colon E\to R$ is a morphism of $A$-ring spectra, then $f_*(t)\in R^2_A(\bC P(\cU_A),\bC P(\eps))$ is an orientation for $R$.
		\item
			As $\res_B^A\colon \Sp_A\to \Sp_B$ is monoidal, the $B$-spectrum $E_B \coloneqq \res_B^A E$ inherits the structure of a commutative $B$-ring spectrum.
			The class $\res_B^A(x(\eps))\in E_B^2(\res_B^A\bC P(\cU_A), \res_B^A\bC P(\eps))$ defines an orientation of $E_B$.
	\end{enumerate}
\end{lemma}

\section{Grassmannians and the Thom Isomorphism}\label{section:2}

We recall the calculation of the (co)homology of equivariant Grassmannians.
Non-equivariantly, the Atiyah-Hirzebruch spectral sequence enables this calculation.
The equivariant version however is more intricate.
Early progress is made in Cole's thesis \cite{C96}, wherein he calculates the homology of $\bC P(\cU_A)$. 
In \cite{CGK02}, Cole, Greenlees, and Kriz find the homology of $\Gr_n(\cU_A)$ for any $n$.
This section serves to recall these results and obtain insights needed for later calculations.

\subsection{The Homology of Projective Space}
Cole \cite{C96} shows that $E\otimes \Sigma^\infty_+ \bC P(\cU_A)$ splits into a coproduct of shifts of $E$, using a filtration of $\bC P(\cU_A)$ into finite projective spaces $\bC P(V)$.
We now give a specific statement of his result.
\begin{definition}
	A \emph{flag} of a complete universe $\cU_A$ is a sequence $\cF = \{V_i\}_i$ of increasing subrepresentations of $\cU_A$ such that the orthogonal complements $ V_{i} - V_{i-1}$ are of dimension one, and $\bigcup_i V_i = \cU_A$. 
	Let $s(\cU_A)$ denote the poset of finite subrepresentations of $\cU_A$ ordered by inclusion. 
	A flag induces a functor $\cF\colon \bN \to s(\cU_A)$ via $\cF(i) = V_i$, which is cofinal.
\end{definition}

\begin{theorem}[{\cite{C96}}]\label{theorem:coles-theorem}
	Given a flag $\cF = \{V_i\}_i$ of $\cU_A$ and an oriented spectrum $E$, the sequence of Chern classes $y(V_i)\in E^{2i}_A(\bC P(\cU_A)_+)$ yields an equivalence of $E_A^*$-modules
	\[
		E^{*}_A(\bC P(\cU_A)_+)\cong \prod_{i\geq 0} E^*_A\{y(V_i)\}.
	\]
	Further, the Kronecker pairing is nonsingular, and we denote by $\beta_i^\cF$ the Kronecker duals of the $y(V_i)$.
	There is an equivalence of $E_*^A$-modules
	\[
		E_{*}^A(\bC P(\cU_A)_+)\cong \bigoplus_{i\geq 0} E_*^A\{\beta_i^\cF\}.
	\]
\end{theorem}

\subsection{Coaugmentations}
Before moving on to higher Grassmannians, we discuss restrictions of the orientation.
\begin{definition}
	An irreducible $A$-representation $\alpha$ gives rise to a unique (unbased) homotopy class $c(\alpha)\colon \bC P(\alpha) \to \bC P(\cU_A)$.
	For an $A$-spectrum $E$, this induces maps in (co)homology
	\[
		\theta(\alpha) \colon E_A^{*}(\bC P(\cU_A)_+)\xlongrightarrow{(c(\alpha)_+)^*} E^*_A, \quad \vartheta(\alpha)\colon E_*^A\xlongrightarrow{(c(\alpha)_+)_*}  E_*^A(\bC P(\cU_A)_+),
	\]	
	respectively referred to as the \emph{$\alpha$-augmentation} and the \emph{$\alpha$-coaugmentation}.
	If $E$ is an $A$-ring spectrum, we write $\vartheta_{\alpha} = \vartheta(\alpha)(1)\in E_*^A(\bC P(\cU_A)_+)$, which is referred to as the \emph{$\alpha$-coaugmentation class}.
\end{definition}

If $E$ is complex oriented, we will find an expression for $\vartheta_\alpha$ with respect to the above basis.

\begin{lemma}[\cite{CGK00}]\label{lemma:restrictions}
	Working with an orientable $A$-spectrum $E$, let $\alpha$ and $\beta$ be characters of $A$.
	The $\alpha$-augmentation of $y(\beta)$ is the Euler class of $\alpha^{-1}\beta$:
	\[
		\theta(\alpha)(y(\beta)) =e(\alpha^{-1}\beta).
	\]
\end{lemma}
\begin{proof}
	Up to homotopy, the map of based topological $A$-spaces 
	\[
		c(\alpha\beta^{-1}) \colon S^0 \to \cof(\bC P(\eps) \to \bC P(\cU_A))
	\]
	factors in two ways. One factorization is 
	\[
		S^0\xlongrightarrow{ c(\alpha\beta^{-1})} \cof(\bC P(\eps)\to \bC P(\alpha\beta^{-1}\oplus \eps ))\xlongrightarrow{\incl} \cof(\bC P(\eps)\to \bC P(\cU_A)),
	\]
	another is given by
	\[
		S^0\xlongrightarrow{c(\alpha)_+} \bC P(\cU_A)_+\xlongrightarrow{m_{\beta^{-1}}}\bC P(\cU_A)_+ \to \cof(\bC P(\eps)\to \bC P(\cU_A)).
	\]
	Pulling back $x(\eps)\in E^*_A(\bC P(\cU_A),\bC P(\eps))$ along these maps, by \cref{lemma:thom-class-two-ways}, the first exhibits the pullback as the Euler class $e(\alpha^{-1}\beta)$, and the second map yields $(c(\alpha)_+)^*(y(\beta)) = \theta(\alpha)(y(\beta))$.
\end{proof}

\begin{example}\label{example:example}
	Let $A = C_2$ be the group with two elements, and let $\sigma$ denote the nontrivial irreducible representation, the sign representation on $\bC$. 
	A flag of the complete $C_2$-universe $(\eps\oplus \sigma)^{\oplus \bN}$ is given by the sequence  $\cF = \{\eps, \sigma\oplus \eps, \eps\oplus\sigma\oplus\eps,\cdots\}$, which alternates between adding either irreducible.
	The associated generators of $E^*_A(\bC P((\eps\oplus \sigma)^{\oplus \bN})_+)$ are given by 
	\[
		1, y(\eps), y(\sigma)y(\eps), y(\eps)y(\sigma)y(\eps), \cdots.
	\]
	We calculate the maps $\theta(\eps)$ and $\theta(\sigma)$ in terms of this basis using $E^*_A$-linearity, multiplicativity, and \cref{lemma:restrictions} for $y(\eps)$ and $y(\sigma)$.
	Recall that for Euler classes, one finds $e(\eps\oplus V) = 0$ for any  $V$.
	\begin{align*}
		\theta(\eps)\colon E^{*}_A(\bC P(\cU_A)_+)&\to E^{*}_A & \theta(\sigma) \colon E^{*}_A(\bC P(\cU_A)_+)&\to  E^{*}_A\\
		y(\eps) &\mapsto 0 & y(\eps)&\mapsto e(\sigma^{-1})\\
		y(\sigma)y(\eps) &\mapsto 0 & y(\sigma)y(\eps) &\mapsto 0\\[-8pt]
		&\hspace{7pt} \vdots & &\hspace{7pt}\vdots
	\end{align*}
	As a final step, we can calculate the coaugmentations by virtue of the duality of the bases in homology and cohomology exhibited in \cref{theorem:coles-theorem}:
	\[
		\vartheta_\eps= \beta_0^\cF, \quad   \vartheta_\sigma = \beta_0^\cF + e(\sigma^{-1}) \beta_1^\cF.
	\]
\end{example}

We can also say what the coaugmentation classes look like in general.
\begin{lemma}\label{lemma:general-coaugmentations}
	Let $\cF = \{V_i\}_i$ be a flag of $\cU_A$. Arguing as in the example above, we find that
	\[
		\vartheta_\alpha = \beta_0^\cF + \sum_{i\geq 1}^\infty e(\alpha^{-1}\otimes V_i)\beta_i^\cF.
	\]
	The sum will be finite, since $\alpha$ must appear as an isotypic component of the $V_i$ for $i\gg 0$, from which point in the sum all Euler classes will vanish.
\end{lemma}

\subsection{Grassmannians}
Classically, there is a CW structure on Grassmannians given by the \emph{Schubert cells}; see e.g. \cite[Ch. 6]{MS74}.
In the equivariant setting, a similar analysis is conducted in \cite[Sec. 3]{CGK02} to find:
\begin{theorem}[{\cite[Thm. 2.2.]{CGK02}}]\label{theorem:grassmannian-cohomology}\label{corollary:duality-grnu}
	The composite of the K\"unneth map with the map induced by classifying the direct sum of bundles
	\[
		E_*^A(\bC P(\cU_A)_+)^{\otimes n}\xrightarrow{\text{K\"unneth}} E_*^A(\bC P(\cU)^n_+)\xrightarrow{\oplus} E_*^A(\Gr_n(\cU_A)_+)
	\]
	exhibits the homology of the Grassmannian as the $n$-symmetrization of the homology of projective space
	\[
		E_*^A(\Gr_n(\cU_A)_+)\cong \Sym_n(E_*^A(\bC P(\cU_A)_+))\cong (E_*^A(\bC P(\cU_A)_+)^{\otimes n})_{\Sigma_n}.
	\]
	
	In the above situation, the map $\oplus$ in cohomology along with the K\"unneth map induces an equivalence to the $\Sigma_n$-invariants of the completed $n$-fold tensor product
	\[
		E_A^*(\Gr_n(\cU_A)_+)\cong \big(E_A^*(\bC P(\cU_A)_+)^{\widehat \otimes n}\big)^{\Sigma_n}.
	\]
	The completion is with respect to the filtration $\{E^*_A(\bC P(V_n)_+)\}_n$ of $E_A^*(\bC P(\cU_A)_+)$ induced by $\cF$.
	Further, the Kronecker pairing of homology and cohomology is an equivalence:
	\[
	E_A^*(\Gr_n(\cU_A)_+)\cong \Mod_{E_*^A}(E_*^A(\Gr_n(\cU_A)_+),E_*^A).
	\]
\end{theorem}

From the proof of the above result in \cite{CGK02}, we find that
\begin{proposition}\label{proposition:sym-ring-multiplication}
	We will use the identification $E_*^A(\Gr_n(\cU_A)_+)\cong \Sym_n(E_*^A(\bC P(\cU_A)_+))$.
	Given an $A$-equivariant isometry $k\colon \cU_A^{\oplus 2}\to \cU_A$, the map
	\[
		\Gr_n(\cU_A)\times \Gr_m(\cU_A)\xrightarrow{\oplus} \Gr_{n+m}(\cU_A\oplus \cU_A)\xrightarrow{\Gr_n(k)} \Gr_{n+m}(\cU_A)
	\]
	induces the following morphism in homology: 
	\[
		\otimes \colon \Sym_n(E_*^A(\bC P(\cU_A)_+))\times \Sym_m(E_*^A(\bC P(\cU_A)_+))\to \Sym_{n+m}(E_*^A(\bC P(\cU_A)_+)).
	\]
	In particular, the map $\alpha\oplus -\colon \Gr_n(\cU_A) \to \Gr_{1+n}(\alpha\oplus \cU_A)\to \Gr_{1+n}(\cU_A)$ corresponds to 
	\[
		\vartheta_\alpha \otimes - \colon \Sym_n(E_*^A(\bC P(\cU_A)_+))\to \Sym_{1+n}(E_*^A(\bC P(\cU_A)_+)).
	\]
\end{proposition}
\begin{proof}
	The second claim follows from the more general statement by factoring the map $\alpha\oplus -$ as
	\[
		\Gr_n(\cU_A)\cong \bC P(\alpha)\times \Gr_n(\cU_A) \xrightarrow{\incl\times \id} \bC P(\cU_A)\times \Gr_n(\cU_A)\xrightarrow{\Gr_n(k)\circ \oplus} \Gr_{1+n}(\cU_A).
	\]
\end{proof}

\subsection{The Thom Isomorphism}

We lay out the definition of Thom classes in the $A$-equivariant setting, which can already be found in \cite[Def. 1.2.]{O82}.
In particular, for oriented $A$-spectra, there are preferred Thom classes for the universal $A$-vector bundles, and these are related multiplicatively. 
When considering bundles over the point, i.e. representations, the Thom classes defined below agree with the classes $\chi(V)$ from \cref{definition:y-alphas}.

\begin{theorem}\label{theorem:orientation-yields-thom-classes}
	Let $(E,x(\eps))$ be an oriented $A$-spectrum with Thom class $t\in E_A^2(\bC P(\cU_A)^{\gamma_1})$, as in \cref{definition:thom-class}.
	One can set $t_1 \coloneqq t$ and define \emph{higher Thom classes} $t_n\in E^{2n}_A(\Gr_n(\cU_A)^{\gamma_n})$ satisfying a multiplicativity and normalization property:
	Let 
	\[
		\Th(\oplus)\colon \Gr_n(\cU_A)^{\gamma_n} \sm  \Gr_m(\cU_A)^{\gamma_m} \to \Gr_{n+m}(\cU_A)^{\gamma_{n+m}}
	\]
	be the Thomification of the map summing $n$- and $m$-planes.
	Note that for a closed subgroup $B\leq A$, a $B$-fixed point $i_W\colon \point \to \res_B^A\Gr_n(\cU_A)$ corresponds to the $B$-representation $i_W(\point) = W$.
	The associated higher Thom classes satisfy the following for $n\geq 1$:
	\begin{enumerate}[label = (\roman*)]
		\item 
			One has
			\[
				\Th(\oplus)^*(t_{n+m}) = t_n\sm t_m,
			\]
		\item
			Restricting $(E,x(\eps))$ to $B$, we obtain the Thom class $\chi^B(W)$  from $x^B(\eps)\coloneqq \res_B^A(x(\eps))$.
			It now holds that 
			\[
				i_W^*(\res_B^A(t_n)) = \chi^B(W) \in E_B^*(S^W).
			\]
	\end{enumerate}
\end{theorem}

\begin{proof}
	The sequence
	\[
		\Gr_{n-1}(\cU_A)\xrightarrow{-\oplus \eps} \Gr_n(\cU_A)\xrightarrow{s_0}\Gr_n(\cU_A)^{\gamma_n}
	\]
	is equivalent to the cofiber sequence $S(\gamma_n)\to D(\gamma_n)\to \Gr_n(\cU_A)^{\gamma_n}$.
	By the analysis of the homology of Grassmannians, it follows that $E^{2n}_A(s_0)$ is injective and that it has $y(\eps)^{\otimes n}$ in its image.
	We define $t_n$ as the unique class such that $s_0^*(t_n) = y(\eps)^{\otimes n}$. 
	Multiplicativity of the Thom classes is immediate from this definition. 
	Again by multiplicativity, we reduce the second claim to the corresponding claim for $n = 1$, where it is true by \cref{lemma:thom-class-two-ways}.
\end{proof}

The relative cap product with a Thom class yields a map relating the (co)homology of the Thom space of a bundle to the base space.

\begin{proposition}[{\cite[X.5.3.]{LMS86}}]\label{proposition:thom-isomorphism}
	For a map from an $A$-CW complex $f\colon X\to \Gr_n(\cU_A)$, we obtain a bundle $\xi = f^*\gamma_n$ over $X$ and a cohomology class $f^*t_n\in E^{2n}_A(X^\xi)$ over the Thom space $X^\xi$. The Thom diagonal together with the cap or cup product now produce $E_*^A$-module isomorphisms
	\[
		\cT^\xi\colon E^{*}_A(X_+)\to  {E^{*+2n}_A}(X^\xi), \quad \cT_\xi\colon   {E_{*+2n}^A}(X^\xi) \to {E_*^A}(X_+).
	\]
\end{proposition}
\begin{proof}
	By passing to (homotopy) colimits, we may reduce this to the case of orbits $X = A/B$ for $B\leq A$ a closed subgroup. The map $f\colon X\to \Gr_n(\cU_A)$ is adjoint to a $B$-map 
	\[
		\hat f\colon \point\to \res_B^A \Gr_n(\cU_A)\simeq \Gr_n(\cU_B).
	\]
	In this case, the normalization property of the Thom classes yields the claim.
\end{proof}

\section{Bordism Spectra}\label{section:3}
It is tom Dieck \cite{D70} who first defines $\MU_G^*$ as an equivariant cohomology theory for a compact Lie group $G$. 
We give one construction of the genuine $G$-spectrum $\MU_G$ representing it.
Along the way, we will introduce the geometric $G$-equivariant bordism spectrum $\mU_G$.
The earliest mention of (the orthogonal version of) this $G$-spectrum can be found in \cite{CW89}.
In upcoming work \cite{B26} it will be shown that for groups with a central identity component, $\mU_G$ represents a certain bordism theory $\Omega^{\mathrm{U}}_G$ of closed $G$-manifolds with \emph{tangential stably almost complex structure}, see also \cref{remark:bordism}.
It is further closely related to the homotopical bordism spectrum $\MU_G$, which can be expressed as a telescopic localization of $\mU_A$, see \cref{proposition:MU-from-mU}.
We make use of the category $\Sp(\cU_G)$ of unitary $G$-prespectra, in which we define models of the spectra of interest.
Then, we pass to the $\infty$-category of genuine $G$-spectra $\Sp_G$.

\subsection{Unitary $G$-Prespectra}
\begin{definition}
	Let $\cU_G$ be a complete $G$-universe.
	We denote by $\Sp(\cU_G)$ the category of \emph{unitary $G$-prespectra} indexed by $\cU_G$.
	Leaving the universe implicit, we usually refer to this category simply as \emph{$G$-prespectra}.
	An object $X\in \Sp(\cU_G)$ is given by  
	\begin{enumerate}[label = (\roman*)]
		\item 
		a based $G$-space $X(V)$ for each finite $V\leq \cU_G$,
		\item
		a based $G$-map $\sigma_V^W\colon S^{W-V}\sm X(V)\to X(W)$ for each inclusion $V\leq W\leq \cU_G$.
	\end{enumerate}
	These are subject to two conditions: firstly, $\sigma_V^V\colon S^0\sm X(V)\cong X(V)$ is the unit equivalence of the smash product.
	Secondly,  for $V\leq W\leq Z$, it holds that $\sigma_W^Z\circ \sigma_V^W$ and $\sigma_V^Z$ correspond via the natural isomorphism $S^{Z-W}\sm S^{W-V}\cong S^{Z-V}$.
	A map of $G$-prespectra $f\colon X\to Y$ is given by morphisms $f(V)\colon X(V)\to Y(V)$ for each $V\leq \cU_G$, which are compatible with the structure maps.
\end{definition}
\begin{remark}
	As the underlying \emph{real} $G$-representation of $\cU_G$ is a complete real universe, $\Sp(\cU_G)$ agrees with the notion of $G$-prespectrum introduced in \cite[Def. I.2.1]{LMS86} with indexing system consisting of finite-dimensional complex subrepresentations.
	We will leverage this perspective in the comparison to genuine $G$-spectra below.
\end{remark}

\begin{example}\leavevmode
\begin{enumerate}[label = $\bullet$]
	\item
	The \emph{sphere $G$-prespectrum} $\bS^{\cU_G}$ is given by the assignment $\bS^{\cU_G}(V) = S^V$, and the structure maps are the natural equivalences $S^{W-V}\sm S^{V}\cong S^{W}$.
	\item
	Given a pointed $G$-space $Y$ and a $G$-prespectrum $X$, we define $X\sm A$ with spaces given by $(X\sm A)(V) = X(V)\sm A$, and with structure maps given by $\sigma_{V}^W\sm A$, where $\sigma_{V}^W$ is the structure map of $X$. 
	\item
	The \emph{suspension $G$-prespectrum} of a pointed $G$-space $Y$ is given by $\Sigma^{\infty,\cU_G} Y\coloneqq \bS^{\cU_G}\sm Y$, and we will use the shorthand $\bS^{\cU_G,W}\coloneqq \Sigma^{\infty,\cU_G} S^W$.
	\item
	Given a $G$-prespectrum $X$ and a $G$-representation $V$, we define $\Omega^VX$ by $(\Omega^VX)(W) =\Omega^VX(W)$.
	The structure maps for $W\leq Z$ are as below, where $\omega(z\sm f)= \big(v\mapsto z\sm f(v)\big)$:
	\[
	S^{Z-W}\sm (\Omega^VX)(W)\xrightarrow{\omega} \Omega^V(S^{Z-W}\sm X(W))\xrightarrow{\Omega^V(\sigma_W^Z)} \Omega^V(X(Z)).
	\] 
	\item
	Given a $G$-prespectrum $X$ and a representation $V\leq \cU_G$, we define the \emph{shift desuspension} of $X$ as $\Lambda^VX$ with level spaces 
	\[
		\Lambda^V X (W)\coloneqq 	\begin{cases}
									X(W-V) & \text{ if } V\leq W,\\
									* & \text{ else.}
									\end{cases}
	\]
	The structure maps for $W\leq W'$ are given in the first case as $(W-Z)\leq (W'-Z)$ and as the inclusion of the basepoint in the second. 
	\end{enumerate}
\end{example}

\begin{definition}
	Let $X\in \Sp(\cU_G)$ be a $G$-prespectrum.
	There is an equivariant isometric embedding $\bR^\infty\to \cU_G$ into the underlying real (orthogonal) representation of $\cU_G$.
	Let $H\leq G$ be a closed subgroup.
	If $n\geq 0$, we define the \emph{$n$-th $H$-equivariant homotopy group} of $X$ by
	\[
		\pi_n^H(X)\coloneqq \colim_{V\in s(\cU_G)}[S^{V}\sm S^n, X(V)]^H_*,
	\]
	and define the \emph{$-n$-th $H$-equivariant homotopy group} of $X$ by
	\[
		\pi_{-n}^H(X)\coloneqq \colim_{\bR^n \leq V\in s(\cU_G)}[S^{V-\bR^n},X(V)]^H_*,
	\]
	using the embedding $\bR^\infty\to \cU_G$.  
	Note that these definitions are functorial: A map of $G$-prespectra $f\colon X\to Y$ induces a map by postcomposing a representative with $f(V)$.
\end{definition}

\begin{remark}
	In the above, $\pi_{-n}^H(X)$ is only defined in reference to an equivariant isometric embedding $\bR^\infty\to \cU_G$.
	As the space of such embeddings is contractible, we obtain a homotopy between any two embeddings.
	One can show that this induces an equivalence on the respective notion of negative homotopy group, leading to a preferred such equivalence.
\end{remark}

\begin{definition}
	A morphism $f\colon X\to Y$ of $G$-prespectra is a \emph{$\underline{\pi}_*$-isomorphism} if for all integers $z$ and closed subgroups $H\leq G$, the map $\pi_z^H(f)\colon \pi_z^H(X)\to \pi_z^H(Y)$ is an isomorphism.
\end{definition}
As we will later see, these are the weak equivalences in a model structure on the category $\Sp(\cU_G)$, which localizes to $\Sp_G$.
The following is \cite[Theorem I.6.1]{LMS86}:
\begin{lemma}
	Given a $G$-representation $V$, the assignments above define an adjunction
	\[
	-\sm S^V\colon \Sp(\cU_G)\rightleftarrows \Sp(\cU_G)\colon  \Omega^V(-)
	\]
	with (co)unit $\eta$ and $\varepsilon$ given in each level as $\eta_V$ and $\varepsilon_V$, the $\Top_{G,*}$ (co)unit of the suspension-loop adjunction.
	Further, both $\eta$ and $\varepsilon$ are natural $\underline{\pi}_*$-isomorphisms.
\end{lemma}

We construct various Thom $G$-prespectra and thank the referee for suggesting the following presentation. 

\begin{construction}\label{construction:thom-spectra}
	Given a $G$-representation $V$, we define a collection of based $G$-spaces, which we will later assemble into unitary spectra.
	The central space in this is given by 
	\[
		\MUP(V) = \Gr(V\oplus V)^\gamma.
	\]
	It is the Thom space over the union of Grassmannians of $n$-planes in $V\oplus V$ with respect to the tautological bundle of non-constant rank.
	Consider the following subspaces of $\MUP(V)$:
	\begin{align*}
		\MUP(V) &= \{(u,U)\mid u\in U\leq V\oplus V\}\cup \{\infty\}\\
		\MUP^{\bfU,[k]}(V) &\coloneqq \{(u,U)\mid |U|-|V| = k \}\cup \{\infty\}\\
		\MU(V) &\coloneqq \MUP^{\bfU,[0]}(V) = \{(u,U)\mid |U|=|V|\}\cup \{\infty\}\\
		\mUP(V) &\coloneqq \{(u,U)\mid U\leq V\oplus V^G\}\cup \{\infty\}\\
		\mU(V) &\coloneqq \{(u,U)\mid U\leq V\oplus V^G \text{ and } |U| = |V|\}\cup \{\infty\}\\
		\MRep(V)&\coloneqq \{(u,U)\mid V\oplus 0\leq U\}\cup \{\infty\}\\
		\MGr(V)&\coloneqq \{(u,U)\mid U\leq V\oplus 0\}\cup \{\infty\}
	\end{align*}
	These assignments are functorial in equivariant linear isometries. 
	For a second representation $W$, we obtain a map
	\[
		\mu(V,W)\colon \MUP(V)\sm \MUP(W)\to \MUP(V\oplus W), \quad ((u,U),(z,Z))\mapsto (\kappa_{V,W}(u + z), \kappa_{V,W}(U + Z)),
	\]
	where $\kappa_{V,W}\colon V\oplus V\oplus W\oplus W\to V\oplus W\oplus V\oplus W$ is the swap map given by $\kappa_{V,W}(v,v',w,w')=(v,w,v',w')$.
	Note that $\mu$ is associative and commutative up to isomorphism.
	There is a further map 
	\[
		\iota(V)\colon S^V\to \MUP(V), \quad  v\mapsto (v,(V\oplus 0)).
	\]
	Combining the two, consider an inclusion of representations $V\leq W$.
	We now define
	\[
		\sigma_V^W\colon S^{W-V}\sm \MUP(V)\xrightarrow{\iota_{W-V}\sm \id } \MUP(W-V)\sm \MUP(V)\xrightarrow{\mu} \MUP(W-V\oplus V)\cong \MUP(W).
	\]
	This structure map can be understood in the following way:
	\[ 
		\sigma_V^W(w, (u,U))=  (w + u,  ((W-V)\oplus 0)+U)
	\]
	where $+$ is the internal sum of vectors or planes, and $(W-V)\oplus 0$ lies in the first summand of $W\oplus W$.
	Note that replacing $\MUP(-)$ with either of the above assignments produces  well-defined maps $\mu(V,W)$, $\iota(V)$, and $\sigma_V^W$ in every case.
\end{construction}

\begin{definition}
	The \emph{periodic unitary Thom $G$-prespectrum} $\MUP^{\cU_G}\in \Sp(\cU_G)$ is given in level $V$ as $\MUP(V)$, and for an inclusion $V\leq W$, the structure map is given by $\sigma_V^W$.
	We obtain a family of sub-prespectra with level spaces as above and restricted structure maps.
	These spectra are denoted by $\MUP^{[k],\cU_G}$, $\MU^{\cU_G}$, $\mUP^{\cU_G}$, $\mU^{\cU_G}$, $\MGr^{\cU_G}$, and $\MRep^{\cU_G}$.
	We further obtain a map from the sphere $\iota\colon \bS^{\cU_G}\to \MUP^{\cU_G}$, factoring through all but $\MUP^{[k],\cU_G}$ for $k\neq 0$.
	\end{definition}

\begin{remark}\label{remark:thom-spectra-map}
	Let us briefly summarize how the $G$-prespectra above correspond.
	\[
		\begin{tikzcd}
		{\MGr^{\cU_G}} & {\mUP^{\cU_G}} & {\mU^{\cU_G}} \\
		{\MRep^{\cU_G}} & {\MUP^{\cU_G}} & {\MU^{\cU_G}}
		\arrow[hook, from=1-1, to=1-2]
		\arrow[hook', from=1-2, to=2-2]
		\arrow[from=1-3, to=1-2, hook']
		\arrow[hook, from=2-1, to=2-2]
		\arrow[hook', from=1-3, to=2-3]
		\arrow[from=2-3, to=2-2,  hook']
		\end{tikzcd}
	\]
	There is a tautological bundle $\gamma$ on $\Rep_G\coloneqq \Gr(\cU_G)$, which we consider alongside the virtual bundle $-\gamma$, leading to the Thom spectra $\MRep^{\cU_G}$ and $\MGr^{\cU_G}$, respectively.
	The map representing the sum of bundles equips $\Rep_G$ with the structure of a $G$-monoid object.
	An element in $\pi^G_0(\Rep_G)$ corresponds to an isotopy class of $G$-representation.
	Inverting either only the trivial representation or all representations yields $G$-monoids $\bUP_G$ and $\BUP_G$.
	As the monoid $\Rep_G$ is itself $\bN$-graded, these inherit a $\bZ$-grading.
	The two spectra $\mUP^{\cU_G}$ and $\MUP^{\cU_G}$ are Thom spectra over these spaces, respectively.
	Finally, the aforementioned grading yields a decomposition of $\MUP^{\cU_G}$ into $\MUP^{[k],\cU_G}$ and of $\mUP^{\cU_G}$ into $\mUP^{[k],\cU_G}\coloneqq \MUP^{[k],\cU_G}\cap \mUP^{\cU_G}$, whose $0$-th components are $\MU^{\cU_G}$ and $\mU^{\cU_G}$.
\end{remark}

\begin{remark}\label{remark:bordism}
	Let $\bU_G$ be the identity component of the space $\bUP_G$ defined in the above \cref{remark:thom-spectra-map}.
	Let $\bO_G$ be the real analogue of this space, constructed from the spaces $\mathrm{B}\mathrm{O}(n)_G$.
	The unitary $G$-spectrum $\mU^{\cU_G}$ represents a (genuine) $G$-spectrum $\mU_G$ detailed in \cref{subsection:comparison-to-genuine}.
	The associated homology theory is the recipient of a Thom-Pontryagin map 
 	\[
		\theta\colon \Omega^{\mathrm{U}}_{G,*}\to \mU_{G,*},
	\]
	which is an equivalence for compact Lie groups with a central identity component \cite{B26}.
	The domain is the bordism ring considered in \cite{Com96} of \emph{tangentially stably almost complex $G$-manifolds}, smooth manifolds $M$ along with a lift of the map $\tau M \colon M\to \bO_G$ along the map forgetting the complex structure: $\bU_G\to \bO_G$.
\end{remark}

\subsection{Ring Structures and Change of Universe}

\begin{definition}
	Given $G$-prespectra $X\in \Sp(\cU_G)$ and $Y\in \Sp(\cU_G')$ for $G$-universes $\cU_G$ and $\cU_G'$, their \emph{external smash product} $X\sm Y\in \Sp(\cU_G\oplus \cU_G')$ has level spaces $ (X\sm Y)(V\oplus V') = X(V)\sm Y(V')$ and structure maps for $V\oplus V'\leq W\oplus W'$ given by the composite
	\[
		S^{W\oplus W' - (V\oplus V')}\sm (X\sm Y)(V\oplus V')\cong S^{W-V}\sm X(V)\sm S^{W'-V'}\sm Y(V')\xrightarrow{\sigma^X\sm \sigma^Y}X(W)\sm Y(W') = (X\sm Y)(W\oplus W').
	\]
\end{definition}

Though this definition is fairly natural, we pass to a new universe in the process.
We now elaborate on how to change the indexing universe of a $G$-prespectrum and comment on the uniqueness of such change functors.

\begin{definition}[{\cite[Sect. II.§1]{LMS86}}]
	Given two complete $G$-universes $\cU_G$ and $\cU_G'$ together with an equivariant isometric embedding $i\colon \cU_G\to \cU_G'$, we obtain a pullback functor given by
	\[
		\Sp(\cU_G') \xrightarrow {i^*} \Sp(\cU_G), \quad i^*(X)(V) =  X(i(V)).
	\]
	Similarly, we obtain a prolongation functor defined by
	\[
		\Sp(\cU_G) \xrightarrow {i_*} \Sp(\cU_G'), \quad i_*(Y)(V') = Y(W)\sm S^{V'-i(W)},\quad  W \coloneqq i^{-1}(V').
	\]
	Further, there is an adjunction $i_*\dashv i^*$.
\end{definition}

Different embeddings give rise to different change of universe functors.
However, this difference disappears at the level of the homotopy category, which is the homotopy category associated to the following model structure.
\begin{proposition}[{\cite[Thm. III.4.2.]{MM02}}]
	There is a \emph{$G$-stable model structure} on the category $\Sp(\cU_G)$ of $G$-prespectra, whose weak equivalences are the $\underline{\pi}_*$-isomorphisms. 
\end{proposition}

\begin{proposition}[{\cite[Thm. II.1.7]{LMS86}}]
Let $i, j\colon \cU_G\to \cU_G'$ be equivariant isometric embeddings of complete universes.
\begin{enumerate}[label = (\roman*)]
	\item
	There is a natural equivalence $i_*\simeq j_*\colon \ho(\Sp(\cU_G))\to \ho(\Sp(\cU_G'))$.
	\item
	There is a natural equivalence $i^*\simeq j^*\colon \ho(\Sp(\cU_G'))\to \ho(\Sp(\cU_G))$.
	\item
	Both $i^*$ and $i_*$ are inverse equivalences on the homotopy category.
\end{enumerate}
\end{proposition}
In particular, any choice of equivariant isometric embedding $i\colon \cU_G\oplus\cU_G\to \cU_G$ leads to a smash product bifunctor $i_*\circ \sm \colon \Sp(\cU_G)\times \Sp(\cU_G)\to \Sp(\cU_G)$.
On the homotopy category, this is independent of $i$ and defines a symmetric monoidal structure on $\ho(\Sp(\cU_G))$.
\begin{remark}
	For universes $\cU_G$ and $\cU_G'$, the two spectra $\MUP^{\cU_G}$ and $\MUP^{\cU_G'}$ are identified in the homotopy category, and the same holds for each of the subprespectra.
\end{remark}

\begin{definition}
	A \emph{$G$-ring prespectrum} consists of a $G$-prespectrum $R$ together	\begin{enumerate}[label = (\roman*)]
		\item 
			a \emph{unit map} $1\colon \bS^{\cU_G}\to R$ in the homotopy category, 
		\item 
			a \emph{multiplication map} $\mu\colon R\sm R\to R$ in the homotopy category,
	\end{enumerate}
	such that left- and right-unitality as well as associativity hold.
	A $G$-ring prespectrum $E$ is \emph{commutative} if $\mu =  \mathrm{swap}\circ \mu$. 
	A \emph{morphism of $G$-ring prespectra} from $R$ to $S$ is a map of $G$-prespectra $f\colon R\to S$ such that $f\circ 1_R = 1_S$ and $f\circ \mu_R  = \mu_S\circ (f\sm f)$.
\end{definition}

\begin{lemma}\label{lemma:ring-prespectra}
	The maps $\mu\colon \MUP(V)\sm \MUP(W)\to \MUP(V\oplus W)$ and $\iota\colon S^V\to \MUP(V)$ defined in \cref{construction:thom-spectra} give $\MUP^{\cU_G}$ the structure of a $G$-ring prespectrum. 
	Further:
	\begin{enumerate}[label = (\roman*)]
		\item
			Each of the Thom spectra $\MU^{\cU_G}$, $\mUP^{\cU_G}$, $\mU^{\cU_G}$, $\MRep^{\cU_G}$, and $\MGr^{\cU_G}$ obtain the structure of a commutative $G$-ring prespectrum with multiplication and unit restricted from the above.
		\item
			Each of the natural inclusions between these $G$-prespectra are morphisms of $G$-ring prespectra. \qed
	\end{enumerate}
\end{lemma}

\subsection{Comparison to Genuine $G$-Spectra}\label{subsection:comparison-to-genuine}
Taking the nerve of the category $\Sp(\cU_G)$, we obtain an $\infty$-category which we may Dwyer-Kan localize at the $\underline{\pi}_*$-isomorphisms.
There is now a comparison functor to genuine $G$-spectra:
\begin{proposition}
	There is an equivalence $L\colon N(\Sp(\cU_G))[\{\text{$\underline{\pi}_*^G$-isos}\}^{-1}]\to \Sp_G$ such that:
	\begin{enumerate}[label = (\roman*)]
		\item 
		$\ho(L)$ is strong monoidal,
		\item
		For a pointed $G$-space $X$, there is a natural equivalence $L(\Sigma^{\infty,\cU_G} X)\simeq \Sigma^\infty l(X)$,
		\item
		For an integer $z$, a $G$-prespectrum $Y$, and a closed subgroup $H\leq G$, there is an equivalence $\pi_z^{H,\cU_G}(Y)\to \pi_z^H(L(Y))$,
		\item
		The image of a commutative $G$-ring prespectrum is naturally a $G$-ring spectrum.
	\end{enumerate}
\end{proposition}
\begin{proof}
	We define $L$ as the composite of equivalences which are strong monoidal on homotopy categories:
	\[
		N(\Sp(\cU_G))[\{\text{$\underline{\pi}_*^G$-isos}\}^{-1}]\xrightarrow{\Psi} N(G\mathscr{P})[\{\text{$\underline{\pi}_*^G$-isos}\}^{-1}]\xrightarrow{\bP} N(G\mathscr{IS})[\{\text{$\underline{\pi}_*^G$-isos}\}^{-1}]\xrightarrow{L'} \Sp_G.
	\]
	The first functor is a change of indexing-family functor from unitary to real subrepresentations of $\cU_G$, and which by \cite[Prop. I.2.4]{LMS86} yields an equivalence of categories which is monoidal on homotopy categories by inspection.
	The second category is a localization of the category of (real) $G$-prespectra, denoted by $G\mathscr{P}$ in \cite{MM02}.
	The category $G\mathscr{IS}$ is the category of orthogonal $G$-spectra, and both it and $G\mathscr{P}$ have stable model structures with weak equivalences given by an appropriate notion of $\underline{\pi}_*^G$-isomorphism.
	The functor $\bP\colon G\mathscr{P}\to G\mathscr{IS}$ obtained from \cite[Def. III.2.7.]{MM02} is shown to be a Quillen equivalence in \cite[Thm. III.4.16.]{MM02}.
	It is evident that its right adjoint $\bbU$ \cite[Def. 2.7]{MM02} is strong monoidal on homotopy categories; thus, the same holds for $\bP$, as $\ho(\bbU)$ and $\ho(\bP)$ are inverses.
	Finally, we refer to \cite[Sects. C.2.-C.3.]{GM23} for the last map, wherein an even stronger statement of an equivalence as monoidal $\infty$-categories is shown.
	Statement (ii) is part of the analysis conducted there, as the functors $\Psi$ and $\bP$ commute with the $G$-suspension spectrum functor.
	For claim (iii), a class $x\in \pi_z^{\cU_G}(Y)$ is represented by a map from $\bS^{\cU_G}\sm S^n \to Y$ or $\Lambda^{\bR^n}\bS^{\cU_G}\to Y$.
	This then corresponds to a class in $\pi_z^G(L(Y))$ and is an equivalence by the analysis of mapping spaces conducted in \cite[Lem. C.3]{GM23}.
	The last claim is a standard construction for monoidal functors.
\end{proof}

\begin{remark}
	We denote the underlying genuine $G$-spectrum of $\MUP^{\cU_G}$ by changing the $\cU_G$-decoration to a $G$-subscript: $\MUP_G\coloneqq L(\MUP^{\cU_G})$. 
	We proceed similarly for the other Thom spectra.
\end{remark}

\begin{corollary}\label{corollary:ring-structures}
	The ring structures discussed in \cref{lemma:ring-prespectra} lead to the following:
	\begin{enumerate}[label = (\roman*)]
		\item 
			The spectra $\MU_G,\MUP_G,\mU_G,\mUP_G,\MRep_G,\MGr_G$ acquire the structure of $G$-ring spectra.
		\item
			The morphisms $\MRep_G\to \MUP_G$, $ \MGr_G\to \mUP_G$, and $\mUP_G\to \MUP_G$ are $G$-ring maps.
	\end{enumerate}
\end{corollary}

\subsection{Unshifted Thom classes}
Given a representation $V$, we introduce two maps 
\[
	\sigma_V\colon \bS^V\to \MUP_G \text{ and }\tau_V\colon \bS \to \MUP_G\otimes \bS^V,
\]
which are $\RO(G)$-graded inverses of one another.
We will later see that $\MU_A$ is orientable, and the Thom class $\chi(V)\colon \bS^V\to \MUP_A\otimes \bS^{|V|\cdot \eps}$ is closely related to $\sigma_V$ and $\tau_V$, see \cref{lemma:unshifted-thom-classes}.
The Thom class $\chi(V)$ can be thought of as \emph{shifted} by the sphere $\bS^{|V|\cdot \eps}$, justifying the name, which keeps the conventions of \cite[Rem. 6.1.20]{S18}.
We will later see that $\pi_*^G(\MUP_G)$ is the localization of $\pi_*^G(\MGr_G)$ (or $\pi_*^G(\MRep_G)$) at the (inverse) unshifted Thom classes.
\begin{definition}
	Given a representation $V\leq \cU_G$ we define two maps:
	\begin{align*}
		s_V\colon  &S^{V\oplus V} \to \Gr(V\oplus V)^\gamma, \quad && (v,v')\mapsto ((v,v'),(V\oplus V)),\\
		t_V\colon &S^V \to \Gr(V\oplus V)^\gamma\sm S^V, \quad && v \mapsto (0,\{0\})\sm v.
	\end{align*}
	We hence obtain maps $\sigma_V\colon \bS^{\cU_G, V}\to \MRep^{\cU_G}$ as well as $\tau_V\colon \bS^{\cU_G}\to \MGr_G\sm S^V$ in $\Sp(\cU_G)$.
	Passing to $\Sp_A$, the map $\bS^V\to \MRep_G $ represents the \emph{unshifted Thom class of $V$} $\sigma_V\in \pi_0^{G}(\Omega^V\MRep_G)$.
	We denote its image in $\pi_0^{G}(\Omega^V\MUP_G)$ by the same name.
	If $V^G = V$, note that $s_V$ maps to $\mUP^{\cU_G}(V)$ and defines a class $\sigma_V\in \pi_0^{G}(\Omega^V\mUP_G)$.
	The map $\bS\to \MGr\otimes S^V$ represents the \emph{inverse unshifted Thom class of $V$} $\tau_V\in \pi_0^{G}(\MGr_G\otimes \bS^V)$.
	We denote its image in $\pi_0^{G}(\mUP_G\otimes \bS^V)$ or $\pi_0^{G}(\MUP_G\otimes \bS^V)$ in the same manner.
\end{definition}

\begin{lemma}\label{lemma:swap-homotopy}
	Let $V$ be a $G$-representation.
	There is a homotopy through equivariant isometries between the two inclusion maps 
	\[
		\incl_1, \incl_2\colon V\to V\oplus V.
	\]
\end{lemma}	
\begin{proof}
	For $t\in [0,\pi]$, the  assignment $t\mapsto \cos(t)\cdot\incl_1 + \sin(t)\cdot \incl_2$ matches the requirements.
\end{proof}

\begin{lemma}\label{lemma:thom-classes-yoga}
Let $V$ and $W$ be $G$-representations.
\begin{enumerate}[label = (\roman*)]
	\item 
	The class $\sigma_V\in \pi_0^G(\Omega^V\MUP_G)$ is an $\RO(G)$-graded unit with inverse given by the inverse Thom class $\tau_V\in \pi_0^G(\MUP_G\otimes \bS^V)$.
	\item
	If $V$ has a trivial $G$-action, the same holds for $\sigma_V\in \pi_0^G(\Omega^V\mUP_G)$ and $\tau_V\in \pi_0^G(\mUP_G\otimes \bS^V)$.
	\item
	Further, $\tau_V\cdot \tau_W\cong \tau_{V\oplus W}$ and $\sigma_V\cdot \sigma_W\cong \sigma_{V\oplus W}$.
\end{enumerate}
\end{lemma}
\begin{proof}
	For (i), we must show that the composite 
	\[
		\bS\otimes \bS^V\xrightarrow{\tau_V\otimes \sigma_V} (\MUP_G\otimes \bS^V)\otimes\MUP_G\to \MUP_G\otimes \bS^V
	\]
	agrees with $\iota\sm \id$.
	This, however, may be checked at the level of $G$-prespectra.
	The above map is then represented as a map of $\cU_G\oplus \cU_G$-prespectra:
	\[
	 \bS^{\cU_G^{\oplus 2}} \sm S^V\cong \bS^{\cU_G}\sm (\bS^{\cU_G} \sm S^V)\xrightarrow{\tau_V\sm \sigma_V} \MUP^{\cU_G}\sm \MUP^{\cU_G}\xrightarrow{\mu} \MUP^{\cU_G^{\oplus 2}},
	\]
	which is represented in degree $V\oplus V$ by 
	\[
	S^V\sm (S^V\sm S^V) \xrightarrow{t_V\sm s_V}\Gr(V\oplus V)^\gamma\sm S^V\sm  \Gr(V\oplus V)^\gamma\xrightarrow{\mu(V,V)}  \Gr(V\oplus V\oplus V\oplus V)^\gamma\sm  S^V.
	\]
	The composite sends $(v,v',v'')$ to $((v'+v'')\in 0\oplus V\oplus 0\oplus V)\sm v$.
	Conversely, the map $\iota(V\oplus V)\sm S^V$ is given by sending this point to $((v+v')\in V\oplus V\oplus 0\oplus 0)\sm v''$.
	Let us show that these maps are homotopic.
	Firstly, by \cref{lemma:swap-homotopy}, we can assume that the map to the Thom space lands over the correct plane.
	The resulting maps differ only by shuffling the input spheres via $(v,v',v'')\mapsto (v',v'',v)$.
	This shuffling is homotopic to the identity, as $V$ is complex, so (real) even dimensional.
	Moving on to (ii), in the case that $V^G = V$, the level spaces of $\MUP^{\cU_G}$ and $\mUP^{\cU_G}$ agree, and we use the same argument.
	Part (iii), i.e. the multiplicativity of the Thom classes, follows from the fact that
	\[
	 S^{V\oplus W}\xrightarrow{t_V\sm t_W} \Gr(V\oplus V)^\gamma \sm S^V\sm \Gr(W\oplus W)^{\gamma}\sm S^W\xrightarrow{\mu(V,W)} \Gr(V\oplus W\oplus V\oplus W)^\gamma \sm S^{V\oplus W}
	\]
	equals $t_{V\oplus W}$.
	A similar analysis applies to the classes $s_V$ and $\sigma_V$.
	\end{proof}

\section{Theorem A and $\MRep_G$}\label{section:4}

Consider a representation $V\leq \cU_G$.
We define a functor by sending $V\leq W\leq \cU_G$ to the map of Abelian groups
\[
	(\sigma_{W-V}\cdot -) \colon \pi_*^G(\Omega^V\MRep_G)\to  \pi_*^G(\Omega^W\MRep_G),
\]
obtaining a diagram indexed by $s(\cU_G)$.
We will show that the following maps define a cocone:
\[
	\Phi^+_V\colon \pi_*^G(\Omega^V \MRep_G) \xrightarrow{\Omega^V\incl_*}  \pi_*^G(\Omega^V \MUP_G) \xleftarrow[\cong]{\sigma_V \cdot -}\pi_*^A(\MUP_G).
\]

\begin{proposition}\label{proposition:MRep-colimit}
	The maps $b_V$ define an equivalence 
	\[
	 	\colim_{V\in s(\cU_G)}(\Phi_V^+)\colon \colim_{V\in s(\cU_A)} \pi_*^A(\Omega^V \MRep_A) \xrightarrow{\cong} \pi_*^A(\MUP_A).
	\]
\end{proposition}
\begin{proof}
	We will prove the equivalence for $\pi_0^A(-)$.
	The general case follows from this with small adaptations.
	Firstly, the naturality of the diagram above, as well as the naturality of the cocone maps, follows from \cref{lemma:thom-classes-yoga}.
	We show injectivity first, then surjectivity, and work with homotopy groups of $G$-prespectra, where we can more easily pick representatives.
	Consider a class $x\in \pi_0^G(\Omega^V\MRep_G)$ in the kernel of the above map.
	It is represented by a morphism $f\colon S^U\to \Omega^V\MRep^{\cU_G}(U)$, and we assume that $\Omega^V\incl \circ f \colon S^U\to \Omega^V\MUP^{\cU_G}(U)$ represents the zero element.
	By suspending the representative $f$ itself, we may assume that $\Omega^V\incl \circ f$ is already nullhomotopic.
	Passing to the $S^V\dashv \Omega^V$ adjoint, we find that $f'\colon S^U\sm S^V\to \MUP^{\cU_G}(U)$ is null.
	In particular, its suspension $S^{U\oplus U}\sm f'$ will be nullhomotopic as well.
	On one hand, the further composite 
	\[
		\mu(s_U \sm \id)\circ S^{U\oplus U}\sm f'\colon S^{U\oplus U}\sm S^U\sm S^V \to \Gr(U\oplus U\oplus U\oplus U)^\gamma 
	\]
	must then also be null.
	On the other hand, we notice that it factors through $\MRep^{\cU_G}(U\oplus U)$, as both $s_U$ and $f'$ land in $\MRep^{\cU_G}(U)$.
	We claim that the above is in turn homotopic to $\sigma_V\cdot f$, i.e. $f$ gets trivialized in the colimit system.
	Indeed, this class is represented by an adjoint of the above composite, adjoining the second $S^U$ and the $S^V$ factor via the loops-suspension adjunctions, and injectivity follows.
	For surjectivity, we consider a class $y\in \pi_0^G(\MUP_A)$, which is represented by $g\colon S^U\to \MUP^{\cU_G}(U)$.
	The class $\sigma_U\cdot y\in\pi_0^G(\Omega^U\MUP_G)$ corresponds to a map of $G$-prespectra $\bS^{\cU_G,U}\to \MUP_G^{\cU_G}$ induced by 
	\[
		S^{U\oplus U}\sm S^U \xrightarrow{s_U\sm g} \MUP^{\cU_G}(U)\sm \MUP^{\cU_G}(U)\xrightarrow{\mu} \MUP^{\cU_G}(U\oplus U).
	\]
	As this composite already lands in $\MRep^{\cU_G}(U\oplus U)$, we conclude that the map lies in the image of the map $\pi_0^G(\Omega^{U\oplus U}\MRep_G)\to \pi_0^G(\Omega^{U\oplus U}\MUP_G)$, and hence we have a preimage in the colimit.
\end{proof}

We now wish to extend the above diagram of homotopy groups to a diagram of genuine $G$-spectra.
As we are mainly interested in the colimit, we only construct the functor for the restriction to a flag.

\begin{construction}\label{construction:omega+}
	Consider a flag $\cF\colon \bN\to s(\cU_G)$, and write $\alpha_{n+1}\coloneqq V_{n+1}-V_n$.
	Functors from the nerve of $\bN$ to an $\infty$-category are equivalently given by their restriction to objects and (1)-morphisms \cite[03HK]{Ker}.
	A cocone on an $\bN$-diagram is a $\Delta^{1}\times N(\bN)$-shaped diagram which is constant on $\Delta^{\{1\}}\times N(\bN)$.
	The product category is a colimit of squares $(\Delta^1)^2$ glued along parallel sides, and functors from squares consist of those functors on the boundary which make the squares commute in the homotopy category.
	Hence, to define a cocone, it suffices to describe the cocone maps and check commutativity in the homotopy category.
	We define a diagram $\Omega^+\colon N(\bN)\to \Sp_G$ from the nerve of $\bN$ to the $\infty$-category of genuine $G$-spectra by sending $n\to n+1$ to
	\[
		\Omega^{V_n}\MRep_G \xrightarrow{\Omega^{V_n}(\sigma_{\alpha_{n+1}} \cdot -)}\Omega^{V_n}\Omega^{\alpha_{n+1}}\MRep_G\xrightarrow{\simeq} \Omega^{V_{n+1}}\MRep_G,
	\]
	which multiplies with the unshifted Thom class $\sigma_{\alpha_{n+1}}\in \pi_0^G(\Omega^{\alpha_{n+1}}\MRep_G)$.
	By the same rationale, we note that the maps 
	\[
		\Psi_V^+\colon \Omega^V\MRep_G\xrightarrow{\Omega^V \incl} \Omega^V\MUP_G \xleftarrow[\simeq]{(\sigma_{V_n}\cdot -)}\MUP_G
	\]
	define a cocone on the diagram $\Omega^+$, as $\Psi_V^+ = \Psi_W^+\circ \Omega^+(V\leq W)$.
\end{construction}

\begin{proposition}\label{proposition:MRep-colimit}
	The above cocone consists of ring maps and induces an equivalence
	\begin{align*}
		\colim_{n\in \bN} (\Psi_V^+) &\colon \colim_{n\in \bN} \Omega^{V_n}\MRep_G \xrightarrow{\simeq} \MUP_G.
	\end{align*}
\end{proposition}
\begin{proof}
	As the colimit is filtered, we check this on homotopy groups.
	As the effect of $\Omega^+(n\to n+1)$ is multiplication by $\sigma_{\alpha_n}$, this follows by \cref{proposition:MRep-colimit} and the cofinality of $\cF\colon \bN\to s(\cU_G)$.
\end{proof}

\begin{definition}
	There is a disjoint union decomposition as below, which is natural for the structure maps:
	\[
		\MRep(V) \cong \bigvee_{n\geq 0} S^V\sm \Gr_{n}(V)^{\gamma_n}.
	\]
	For $n\geq 0$, the assignment $\M(n) = S^V\sm \Gr_{n}(V)^{\gamma_n}$ yields a sub-prespectrum $\M(n)^{\cU_G}$ of $\MRep^{\cU_G}$.
	We denote the underlying genuine spectrum by $\M(n)_G$.
\end{definition}

We now restrict to the case of an abelian compact Lie group $A$ to consider the homology of $\MRep_A$.
\begin{proposition}\label{proposition:MRep-htp-grps}
	There is an equivalence of $E_*^A$-algebras
	\[
		B\colon \Sym(E_*^A(M(1)_A))\xrightarrow{\cong} E_*^A(\MRep_A).
	\]
\end{proposition}
\begin{proof}
	The natural inclusion map $i_n\colon M(n)^{\cU_G}\to \Sigma^{\infty,\cU_G}\Gr_n(\cU_A)^{\gamma_n}$ is a $\underline{\pi}_*$-isomorphism:
	We can check this on representatives.
	Any map from a sphere to $S^V\sm \Gr_n(\cU_A)^{\gamma_n}$ already factors over $S^V\sm \Gr_n(W)^{\gamma_n}$ for some representation $W$ by compactness of the domain, hence landing in some $S^W\sm\Gr_n(W)^{\gamma_n}$ up to stabilization.
	This proves injectivity on homotopy groups.
	A similar argument will apply to a nullhomotopy of a representative, proving injectivity.
	The multiplication of $\MRep^{\cU_A}$ is induced by maps $\mu_{m,n}\colon \M(n)^{\cU_G}\sm \M(m)^{\cU_G}\to \M(n+m)^{\cU_G}$ given at $V\oplus W$ by
	\[
		\mu(V,W)\colon S^V\sm \Gr_n(W)^{\gamma_n}\sm S^V\sm \Gr_n(W)^{\gamma_n} \to S^{V\oplus W} \sm \Gr_{n+m}(V\oplus W)^{\gamma_{n+m}}.
	\]
	The above is equivalent to the map 
	\[
		\Th(\oplus)\colon \Sigma^{\infty,\cU_A} \Gr_n(\cU_A)^{\gamma_n} \sm  \Sigma^{\infty,\cU_A} \Gr_m(\cU_A)^{\gamma_m}\to \Sigma^{\infty,\cU_A\oplus \cU_A} \Gr_{n+m}(\cU_A\oplus \cU_A)^{\gamma_{n+m}}
	\]
	via the equivalences $i_n\sm i_m$ and $i_{n+m}$.
	Applying $E_*^A$-homology, this yields a map
	\[
		E_*^A(\Th(\oplus))\colon 	E_*^A(\Gr_n(\cU_A)^{\gamma_n})\otimes_{E_*^A}E_*^A(\Gr_m(\cU_A)^{\gamma_m})\to E_*^A(\Gr_{n+m}(\cU_A\oplus \cU_A)^{\gamma_{n+m}}).
	\]
	Using the Thomified version of \cref{theorem:grassmannian-cohomology}, we find that this agrees with
	\[
		\otimes \colon \Sym_n(E_*^A(\M(1)_A))\otimes_{E_*^A}\Sym_m(E_*^A(\M(1)_A))\to \Sym_{n+m}(E_*^A(\M(1)_A)),
	\]
	yielding the desired equivalence of rings.
\end{proof}

Recall that the tautological bundle $\gamma_1$ allows for a Thom isomorphism, whose inverse we denote by:
\[
	\Th^+ \colon E_{*-2}^A(\bC P(\cU_A)_+)\cong E_{*}^A(\M(1)_A) \cong E_*^A(\bC P(\cU_A)^{\gamma_1}).
\]
We write $R_E^+\coloneqq E_{*}^A(\M(1)_A)$ and we make use of the shorthand $u_\alpha^+ \coloneqq \Th^+(\vartheta_\alpha)$.

\begin{theorem}\label{theorem:A}
	The ring homomorphism 
	\[
		\Sym(R_E^+)\xrightarrow[\cong]{B} E_*^A(\MRep_A)\xrightarrow{E_*^A(\incl)} E_*^A(\MUP_A) 
	\]
	induces an equivalence of $E_*^A$-algebras
	\[
		E_*^A(\MUP_A)\cong \Sym(R_E^+)[(u_\alpha^+)^{-1}\mid \alpha\in A^*].
	\]
\end{theorem}
\begin{proof}
	We use \cref{proposition:MRep-colimit} after tensoring with $E$ to obtain the equivalence 
	\[
		\colim_{n\in \bN} E_*^A(\Omega^{V_n}\MRep_A) \cong E_*^A(\MUP_A),
	\]
	noting that homotopy groups commute with filtered colimits.
	We now show that for $V_n\leq V_{n+1}$ of codimension $\alpha\coloneqq \alpha_{n+1}$, the transition map is multiplication with $u_\alpha^+$.
	By multiplicativity, we reduce to the case of $0\leq \alpha$.
	Before tensoring with $E$, the transition map is given by the multiplication $(\sigma_{\alpha}\cdot -) \colon \MRep_A\to \Omega^{\alpha}\MRep_A$.
	In turn, the class $\sigma_{\alpha}\colon \bS^{\alpha} \to \MRep_A$ comes from a map of $A$-prespectra defined in level $\alpha$ as:
	\[
	\bS^{\alpha,\cU_A}(\alpha) = S^{\alpha\oplus \alpha}\xrightarrow{\id \sm \Th(c_{\alpha})} S^\alpha\sm \bC P(\alpha)^{\gamma_1} = \M(1)^{\cU_A}(\alpha)\hookrightarrow \MRep^{\cU_A}(\alpha).
	\]
	Using the identification of $E_*^A(\MRep_A)$ with $\Sym(R_E)$, we see that this corresponds to $u_\alpha^+ = \Th^+(\vartheta_{\alpha_n})$.
	Hence, the colimit agrees with the telescopic localization at the given classes.
\end{proof}

\begin{remark}\label{remark:un-thomified}
	In \cite[Prop. 23.3]{S11}, Strickland proves the corresponding result \emph{prior} to taking Thom spectra for so-called \emph{periodically orientable} $A$-spectra with $A$ finite.
	In the language of the present article, we may rephrase this result as follows:
	Given an $A$-representation $V$, the map classifying the $A$-vector bundle $\gamma_n\oplus V$ is denoted by $\phi(V)\colon \BU(n)_A\to \BU(n+|V|)_A$ .
	For $\Rep_A = \Gr(\cU_A)\simeq \coprod_{n\geq 0}\BU(n)_G$, this yields maps $\phi(V)\colon \Rep_A\to \Rep_A$, whose colimit is given by $\BUP_A \coloneqq \colim_{V\in s(\cU_A)} \Rep_A$, and we conclude that
	\[
		E_*^A((\BUP_A)_+) \cong E_*^A((\Rep_A)_+)[\vartheta_\alpha^{-1}\mid \alpha\in A^*].
	\]
	We can equip the copy of $\Rep_A$ in the $V$-entry of the $s(\cU_A)$-shaped diagram with the (virtual) bundle $\gamma - V$ and $\BUP_A$ with the tautological bundle $\gamma$.
	Passing to Thom spectra then translates this equivalence to the equivalence in the above theorem.
\end{remark}

The class $u^+_\alpha$ lies in $E_*^A(\MUP_A^{[1]})$ and is invertible.
Note also that the multiplication of $\MUP_A$ is graded for the decomposition into the $\MUP_A^{[k]}$.
\begin{lemma}\label{lemma:periodicity}
 	For an irreducible representation $\alpha$, multiplication with $\Th^+(\vartheta_\alpha)$ on $E_*^A(\MUP_A)$ defines an equivalence
	\[
		E_*^A(\MUP_A^{[k]})\to  E_{*-2}^A(\MUP_A^{[k+1]}).
	\]
\end{lemma}
Note that the above can be seen as a periodicity phenomenon for $E\otimes \MUP_A$, which in fact can be derived from a periodicity result of $\MUP_A$ itself, see e.g. \cite[Proposition 6.1.20]{S18}.

The following is a restatement and proof of \cref{corollary:intro-A-coordinatized}.
\begin{corollary}\label{corollary:A-coordinatized}
	Given a flag $\cF = \{V_i\}_i$ of a complete universe, there is an associated $E_*^A$-basis $\{\beta_i^\cF\}_i$ of $E_*^A(\bC P(\cU_A)_+)$, and we write $b_i^+\coloneqq \Th^+(\beta^\cF_i)$, leaving the flag implicit.
	The above then reduces to
	\[
		E_*^A(\MUP_A)\cong E_*^A[b^+_0, b^+_1, b_2^+,\cdots ][(u^+_\alpha)^{-1}\mid \alpha\in A^*].
	\]
	For an element $x\in E_*^A(\MUP_A)$, we write $\widetilde x$ for $(u^+_{V_1})^{-1}\cdot x$.
	Note that $\beta_0^\cF = \vartheta_{V_1}$.
	There is an isomorphism of rings
	\[
		E_*^A(\MU_A)\cong E_*^A[\widetilde{b^+_1}, \widetilde{b^+_2}, \widetilde{b_3^+}, \cdots ][(\widetilde{u^+_{\alpha}})^{-1}\mid \alpha\in A^*].
	\]
\end{corollary}
\begin{proof}
	The presentation of $E_*^A(\MUP_A)$ is immediate.
	Continuing, note that $E_*^A(\MU_A)\subseteq E_*^A(\MUP_A)$ is precisely the subring making up the zero summand in the aforementioned decomposition.
	By multiplying with a power of the class  $u^+_{V_1}$, we can shift the classes $b^+_i$ and $u^+_{\alpha}$ to $\widetilde{b_i^+}$ and $\widetilde{u_{\alpha}^+}$, respectively. 
	Hence, the homology of $\MU_A$ corresponds to the displayed subalgebra.
\end{proof}

The inverted elements $u^+_\alpha$ are not a subset of the polynomial generators $b_i^+$ of $R_E^+$.
Nonetheless, the elements can be rather controllable, as highlighted by the referee:
\begin{example}
	Let $A$ be a finite abelian Lie group with $|A| = n$, and let $\eps = \alpha_1, \alpha_2, \cdots, \alpha_n$ be its irreducible representations.
	We chose the periodic flag $\cF = \{ \eps, \eps\oplus \alpha_2, \cdots , \rho, \rho\oplus \eps, \cdots, \rho^2, \rho^2\oplus \eps, \cdots \}$, where $\rho$ is the regular representation of $K$, the sum of all its irreducibles.
	We use \cref{lemma:general-coaugmentations} to find that $u_{\alpha_j}^+ \in  b_0^+ + E_*^A\{b_i^+\mid 0 < i < n\}$.
	The product $v^+ = \prod_j u_{\alpha_j}^+$ is a polynomial in $b_0$ which is of degree $n$ and monic, lying in the ring $R = E_*^A[b_i^+\mid i < n]$.
	One obtains an equivalence
	\[
		E_*^A(\MUP_A)\cong R[(v^+)^{-1}][b_i^+\mid i > n].
	\] 
\end{example}

\begin{remark}
	Let us contrast with \cite[Thm. 8.2.]{CGK02}.
	The authors Cole, Greenlees, and Kriz define classes in $E_*^A(\MU_A)$ via the composite \cite[Def. 8.1.]{CGK02}:
	\[
		E_{*}^A(\bC P(\cU_A)_+)\xrightarrow{\Th^+}E_{*-2}^A(\M(1)_A) \xrightarrow{(\chi(\alpha)^{-1}\cdot -)}E_{*}^A(\Omega^\alpha \M(1)_A) \xrightarrow{\Psi_\alpha^+} E_*^A(\MU_A).
	\]
	For a flag $\cF = \{V_i\}_i$ with $V_1 = \alpha$, we find that $\beta_i^\cF$ is sent to $\widetilde b_i^+$ along this map.
	We conclude that the statement of \cite[Thm. 8.2.]{CGK02}, that $E_*^A(\MU_A)\cong E_*^A[\widetilde{b_1^+}, \widetilde{b_2^+}, \cdots ]$, is only correct if all $\widetilde{u_\alpha^+}$ are units. 
	By \cref{lemma:general-coaugmentations} this is only the case when all Euler classes vanish.
\end{remark}

\section{Theorem B and $\MGr_G$}\label{section:5}

Conversely to the above construction, we can also obtain $\MUP_G$ as a colimit over copies of the $G$-spectrum $\MGr_G$, which is the Thom spectrum of the \emph{negative} tautological bundle over $\Rep_G=\Gr(\cU_G)$.
We will calculate the homology rings $E_*^A(\MUP_A)$ and $E_*^A(\mUP_A)$.
Both results are of interest: 
Firstly, this second calculation of $E_*^A(\MUP_A)$ proceeds by \emph{negative} Thom spectra.
In the non-equivariant setting, these are considered e.g. in \cite{GMTW09}.
The two identifications of $E_*^A(\MUP_A)$ are related by $\otimes$-inversion, which is applicable for any Thom spectrum. 
Secondly, via the approach of Thom spectra for negative bundles, we can calculate the homology $E_*^A(\mUP_A)$.
As alluded to, the spectrum $\mUP_A$ is of independent interest, relating to the bordism theory tangentially stably almost complex $G$-manifolds, and contains the sub-ring $E_*^A(\mU_A)$ as a summand.

Consider a representation $V\leq \cU_G$.
We can define a map
\[
	\Phi_V^-\colon \pi_*^G(\MGr_G\otimes \bS^V) \xrightarrow{(\incl\sm \id)_*}  \pi_*^G(\MUP_G\otimes \bS^V) \xleftarrow[\cong]{(\tau_V \cdot -)}\pi_*^A(\MUP_A).
\]
Similarly, for any $n\geq 0$, we define the map
\[
	\phi_n^-\colon  \pi_*^G(\MGr_G\otimes \bS^{\eps^n}) \xrightarrow{(\incl\sm \id)_*}  \pi_*^G(\mUP_G \otimes \bS^{\eps^n}) \xleftarrow[\cong]{(\tau_{\eps^n} \cdot -)}\pi_*^A(\mUP_A).
\]

\begin{proposition}\label{proposition:MGr-htp-grps}
	Multiplying with the classes $\tau_{V'-V}$ for $V\subseteq V'$ or simply by $\tau_\eps$, we obtain colimit diagrams as below, and the maps $\Phi_V^-$ and $\phi_n^-$ define equivalences
	\begin{align*}
	\colim_{V\in s(\cU_G)}(\Phi^-_V)&\colon 	\colim_{V\in s(\cU_G)} \pi_*^G(\MGr_G\otimes \bS^V)\xrightarrow{\cong} \pi_*^G(\MUP_G),\\
		\colim_{n\in \bN}(\phi^-_V) &\colon  \colim_{n\in \bN} \pi_*^G(\MGr_A\otimes \bS^{\eps^n}) \xrightarrow{\cong} \pi_*^G(\mUP_G).
	\end{align*}
\end{proposition}
\begin{proof}
	Let us firstly deal with the case of $\MUP_G$.
	We now show bijectivity for the restriction to $\pi_0^G$. 
	The general case follows easily from this.
	For injectivity, consider a class $\alpha\in \pi_0^G(\MGr_G\otimes \bS^V)$ which is mapped to 0 in $\pi^G_0(\MUP_G\otimes \bS^V)$.
	The class $\alpha$ is represented by $f\colon S^W\to \MGr(W)\sm S^V$, and its composite with the inclusion to $\MUP^{\cU_G}\sm S^V$, given by
	\[
		\tilde f\colon S^W\to \Gr(W\oplus W)^\gamma\sm S^V,
	\]
	represents the zero element.
	By stabilizing $f$ in the first place, we assume that $\tilde f$ is already nullhomotopic.
	Multiplying $f$ by $\tau_{W}$, we obtain the composite
	\[
		S^{W}\sm S^{W}\xrightarrow{t_{W}\sm f} \Gr(W\oplus W)\sm S^{W}\sm \MGr(W)\sm S^V\xlongrightarrow{\mu} S^{W}\sm \MGr(W\oplus W)\sm S^V. 
	\]
	On the other hand, as $\tilde f$ is nullhomotopic, we know that $S^{W}\sm \tilde f \colon S^{W\oplus W} \to S^{W}\sm \Gr(W\oplus W)^\gamma\sm S^V $ is nullhomotopic as well,
	but this coincides precisely with the above composite, up to the equivalence 
	\[
		\MGr(W\oplus W)= \Gr(W\oplus W\oplus 0\oplus 0)^\gamma \cong \Gr(W\oplus W)^{\gamma}.
	\]
	For surjectivity, consider a class $\beta\in \pi_0^G(\MUP_G)$ represented by $g\colon S^W\to \Gr(W\oplus W)^\gamma$.
	Multiplying by $\tau_W$, we obtain a map
	\[
		 S^{W\oplus W}\xrightarrow{t_W\sm g} \Gr(W\oplus W)^\gamma \sm S^{W}\sm \Gr(W\oplus W)^\gamma\xrightarrow{\mu} \Gr(W\oplus W\oplus W\oplus W)^\gamma\sm S^W,
	\]
	which already factors through $\Gr(0\oplus W\oplus 0\oplus W)^\gamma \sm \bS^W$.
	Applying \cref{lemma:swap-homotopy}, we conclude that this lies in the image of $\pi_0^G(\MGr_G\sm S^{W})$.
	This proves bijectivity on $\pi_0^G$, which easily generalizes to $\pi_*^G$.
	
	For the colimit formula involving $\mUP$, we use the same proof strategy, multiplying instead with $\tau_{\eps^n}$ for $n = |W^G|$, allowing the same argument to apply.
\end{proof}
\begin{construction}
	Given a flag $\cF\colon \bN\to s(\cU_G)$ with $\cF(n) = V_n$ and $\alpha_{n+1}=V_{n+1}-V_n$, we can define a diagram $\Omega^- \colon N(\bN)\to \Sp_G$ in analogy to the construction \cref{construction:omega+} sending $n\to n+1$ to the map of genuine $G$-spectra
	\[
		\MGr_G\otimes \bS^{V_n} \xrightarrow{(\tau_{\alpha_{n+1}}\cdot -)\otimes \bS^{V_n}}\MGr_G\otimes \bS^{\alpha_{n+1}}\otimes \bS^{V_n}\xrightarrow{\simeq} \MGr_G\otimes \bS^{V_{n+1}},
	\]
	which multiplies with the Thom class $\tau_{\alpha_{n+1}}\in \pi_{0}^G(\MGr_G\otimes \bS^{\alpha_{n+1}})$.
	Again, as discussed in the prior section, we obtain a cocone
	\begin{align*}
		\Psi_V^-\colon \MGr_G \otimes \bS^{V_n}\xrightarrow{\incl\otimes \bS^{V_n}} \MUP_G\otimes \bS^{V_n} \xleftarrow[\simeq]{(\tau_{V_n}\cdot -)}\MUP_G.
	\end{align*}
	As opposed to a flag, we may take $V_n = \eps^n$ and define $\Omega_{\mathrm{geom}}^-\colon \bN\to \Sp_G$, along with a cocone
	\begin{align*}
		\psi_V^-\colon \MGr_G \otimes \bS^{\eps^n}\xrightarrow{\incl\otimes \bS^{\eps^n}} \mUP_G\otimes \bS^{\eps^n} \xleftarrow[\simeq]{(\tau_{\eps^n}\cdot -)}\mUP_G,
	\end{align*}
	with respect to the periodic geometric bordism $G$-spectrum $\mUP_G$.
\end{construction}

\begin{proposition}\label{proposition:MGr-colimit}
	There are equivalences of $G$-ring spectra
	\begin{align*}
	\colim_{n\in \bN}(\Psi_{V_n}^-) &\colon \colim_{n\in \bN} \MGr_G \otimes \bS^{V_n} \xrightarrow{\simeq} \MUP_G,\\
		\colim_{n\in \bN}(\psi_n^-)&\colon \colim_{n\in \bN} \MGr_G \otimes \bS^{\eps_n} \xrightarrow{\simeq} \mUP_G.
	\end{align*}
\end{proposition}
\begin{proof}
	As the colimits are filtered, we check this on homotopy groups.
	There, the effect of $\Omega^-(n\to n+1)$ or $\Omega^-_{\mathrm{geom}}(n\to n+1)$ is multiplication by $\tau_{\alpha_{n+1}}$ or $\tau_{\eps}$, and the claim follows by \cref{proposition:MGr-htp-grps} and the cofinality of the flag.
\end{proof}

Combining the equivalences in \cref{proposition:MGr-colimit} and taking the colimit over the trivial sub-universe first yields a composite equivalence
\[
	\colim_{W\in s(\cU_G), W^G = 0} \pi_*^G(\mUP_G\otimes \bS^W)\xleftarrow{\cong} \colim_{V\in s(\cU_G)}\pi_*^G(\MGr_G\otimes \bS^V) \xrightarrow{\cong}\pi_*^G(\MUP_G).
\]
Consider a cofinal map $\cF'\colon \bN\to s(\cU_G-(\cU_G)^G)$ with $\cF'(n) = V_n'$ and $\alpha_{n+1}' \coloneqq V_{n+1}'-V_n$.
We define a functor $\Omega^-_{\mUP}$ just as $\Omega^-$, with values $\Omega^-_{\mUP}(n) = \mUP_G\otimes \bS^{V_n'}$ and transition maps given by multiplication with $\tau_{\alpha_{n+1}'}$.
A cocone $\Psi^-_{\mUP_G,V_n'}\colon \mUP_G\otimes \bS^{V_n'}\to \MUP_G$ defined as above leads to the following:
\begin{corollary}\label{proposition:MU-from-mU}
	There are equivalences
	\begin{align*}
		\colim_{n\in \bN}(\Psi_{V_n}^-) &\colon \colim_{n\in \bN} \mUP_G \otimes \bS^{V_n} \xrightarrow{\simeq} \MUP_G,\\
	\colim_{n\in \bN}(\Psi_{V_n}^-) &\colon \colim_{n\in \bN} \mU_G \otimes \bS^{V_n'} \xrightarrow{\simeq} \MU_G,
	\end{align*}
	where the latter is a restriction of the former to a summand.
\end{corollary}

\begin{definition}
	There is a disjoint union decomposition as below, which is natural for the structure maps:
	\[
		\MGr(V) \cong \bigvee_{n\geq 0} \Gr_{|V|-n}(V)^{\gamma_{|V|-n}}.
	\]
	For $n\geq 0$, the assignment $\M(-n)(V) \coloneqq \Gr_n(V)^{\gamma^\bot}\cong \Gr_{|V|-n}(V)^{\gamma_{|V|-n}}$ with $\gamma^\bot$ the orthogonal complement bundle defines a sub-prespectrum $\M(-n)^{\cU_G}$ of $\MGr^{\cU_G}$ with  underlying genuine $G$-spectrum $\M(-n)_G$.
\end{definition}
Once again, to calculate homology, we change to an abelian compact Lie group $A$.
\begin{proposition}\label{proposition:MGr-homology}
	There is an equivalence of $E_*^A$-algebras
	\[
		C\colon \Sym(E_*^A(\M(-1)_A))\xrightarrow{\cong} E_*^A(\MGr_A).
	\]
\end{proposition}
\begin{proof}
	We exhibit the subspectrum $\M(-n)_A$ of $\MGr_A$ as the Thom spectrum of the virtual bundle $-\gamma$ over $\Gr_n(\cU_A)$.
	We use the language of \cite[Sect. X.§3]{LMS86}.
	There is a filtration of the base space; for finite $V\leq \cU_A$ we consider $\Gr_n(V)$ and the bundle $\gamma_n^\bot$ over it, for which we have $\gamma_n^\bot\oplus \gamma_n = V$.
	Hence, the virtual bundle $\gamma_n^\bot-V$ corresponds to $-\gamma_n$. 
	The associated Thom prespectrum assigns to $V$ the Thom space of the $V$-part of the filtration, which is precisely $\M(-n)(V) = \Gr_n(V)^{\gamma_n^\bot}$.
	
Returning to the claim of the proposition, note that the multiplication of $\MGr_A$ is induced by maps of $\cU_A\oplus \cU_A$-prespectra $\mu\colon \M(-n)^{\cU_A}\sm \M(-m)^{\cU_A}\to \M(-n-m)^{\cU_A^{\oplus 2}}$ given at $V\oplus W$ by 
	\[
		\Th(\oplus)\colon \Gr_n(V)^{\gamma_n^\bot}\sm \Gr_m(W)^{\gamma_m^\bot}\to \Gr_{n+m}(V\oplus W)^{\gamma_{n+m}^\bot}.
	\]
	On homology, this yields the homomorphism
	\[
	E_*^A(\Th(\oplus)) \colon E_*^A(\M(-n)^{\cU_A}\sm \M(-m)^{\cU_A})\to E_*^A(\M(n+m)^{\cU_A\oplus \cU_A}). \tag{$\dag$}
	\]
	Applying the Thom isomorphism \cite[Theorem X.5.3]{LMS86}, this is in turn equivalent to the map
	\[
	E_*^A(\oplus)\colon E_*^A(\Sigma^{\cU_A,\infty}\Gr_n(\cU_A)_+\sm\Sigma^{\cU_A,\infty}\Gr_n(\cU_A)_+) \to E_*^A(\Sigma^{\cU_A\oplus \cU_A,\infty}\Gr_n(\cU_A\oplus \cU_A)_+),
	\]
	which is in turn equivalent to $	E_*^A(\oplus)\colon E_*^A(\Gr_n(\cU_A)_+)\otimes_{E_*^A}E_*^A(\Gr_m(\cU_A)_+)\to E_*^A(\Gr_{n+m}(\cU_A\oplus \cU_A)_+)$.
	Applying the Thom isomorphism to the equivalence $E_*^A(\Gr_n(\cU_A)_+)\cong \Sym_n(E_*^A(\bC P(\cU_A)_+))$, of \cref{theorem:grassmannian-cohomology}, we find that $E_*^A(\M(-n)_A)\cong \Sym_n(E_*^A(\M(-1)_A))$.
	In this same theorem it is shown that $E_*^A(\oplus)$ agrees with $\otimes\colon \Sym_n(E_*^A(\bC P(\cU_A)_+))\otimes_{E_*^A} \Sym_m(E_*^A(\bC P(\cU_A)_+))\to \Sym_{n+m}(E_*^A(\bC P(\cU_A)_+))$.
	Once again, we apply the Thom isomorphism to identify ($\dag$) with
	\[
		\Sym_n(E_*^A(\M(-1)_A))\otimes_{E_*^A}\Sym_m(E_*^A(\M(-1)_A))\to \Sym_{n+m}(E_*^A(\M(-1)_A)).
	\]
	\end{proof}

As noted in the proof above, the virtual bundle $-\gamma_1$ defined for $\bC P(\cU_A)$ defines the Thom spectrum $\M(-1)_A$ above and allows for a Thom isomorphism, whose inverse we denote by:
\[
	\Th^- \colon E_{*+2}^A(\bC P(\cU_A)_+)\xrightarrow{\cong} E_{*}^A(\M(-1)_A)
\]

We write $R_E^-\coloneqq E_{*}^A(\M(-1)_A)$ and we make use of the shorthand $u_\alpha^- \coloneqq \Th^-(\vartheta_\alpha)$.

\begin{theorem}\label{theorem:B}
	The ring homomorphisms
	\begin{align*}
		\Sym(R_E^-)\xrightarrow[\cong]{C} E_*^A(\MGr_A)\xrightarrow{E_*^A(\incl)} E_*^A(\mUP_A),\\
		\Sym(R_E^-)\xrightarrow[\cong]{C} E_*^A(\MGr_A)\xrightarrow{E_*^A(\incl)} E_*^A(\MUP_A),
	\end{align*}
	induce equivalences of $E_*^A$-algebras
	\begin{align*}
		E_*^A(\mUP_A) & \cong \Sym(R_E^-)[(u_\eps^-)^{-1}],\\
		E_*^A(\MUP_A) & \cong \Sym(R_E^-)[(u_\alpha^-)^{-1}\mid \alpha\in A^*].
	\end{align*}
	
\end{theorem}
\begin{proof}
	The equivalence in \cref{proposition:MGr-colimit} may be tensored with $E$ to obtain the isomorphism 
	\[
		\colim_{V\in s(\cU_A)} E_*^A(\MGr_A\otimes \bS^V) \cong E_*^A(\MUP_A).
	\]
	We show that for $V\leq W\leq \cU_A$ of codimension $\alpha$, the transition map is $(\Th^-(\vartheta_{\alpha})\cdot -)$.
	We reduce to the case of $0\leq \alpha$, where before tensoring with $E$, the transition map is given by $(\tau_\alpha \cdot -)\colon \MGr_A\to\MGr_A\otimes \bS^\alpha$.
	In turn, the class $\tau_\alpha$ is represented by a map of prespectra $\bS^{\cU_A}\to \MGr^{\cU_A}\sm S^\alpha$ induced in level $\alpha$ by
	\[
		\begin{tikzcd}[column sep = huge]
		{S^\alpha} & {\Gr(\alpha)^{\gamma}\sm S^\alpha} & {\M(-1)^{\cU_A}(\alpha)\sm S^\alpha} & {\MGr^{\cU_A}(\alpha)\sm S^\alpha.} \\
	{S^0\sm S^\alpha} & { \bC P(\alpha)^{\gamma^\bot}\sm S^\alpha}
	\arrow["{t_\alpha}", from=1-1, to=1-2]
	\arrow["\cong"', from=1-1, to=2-1]
	\arrow[equals, from=1-2, to=1-3]
	\arrow[hook, from=1-3, to=1-4]
	\arrow["{\Th(c_\alpha)\sm \id}", from=2-1, to=2-2]
	\arrow["{\bot\sm \id}"', from=2-2, to=1-2]
	\arrow["\cong", from=2-2, to=1-2]		\end{tikzcd}
	\]
	The spectrum $\M(-1)^{\cU_G}$ is the Thom spectrum of $-\gamma_1$ over $\bC P(\cU_A)$.
	As seen in the proof of \cref{proposition:MGr-homology}, in level $\alpha$ this corresponds to $\gamma^\bot$ over $\bC P(\alpha)$.
	As $E_*^A(c_\alpha)(1) = \vartheta_\alpha$, we find that in homology, the above composite corresponds to $\Th^-(\vartheta_\alpha)\in E_*^A(\MGr^{\cU_A}\sm S^\alpha)$.
	The colimit is hence given by a telescopic localization at the given classes.
	In the case of $\mUP_A$, the transition maps all correspond to $u^-_\alpha = \Th^-(\vartheta_\eps)$, hence leading to the above result.
	\end{proof}

We now prove \cref{corollary:intro-B-coordinatized} from the introduction.
\begin{corollary}\label{corollary:B-coordinatized}
	Given a flag $\cF = \{V_i\}_i$ of a complete universe, there is an associated $E_*^A$-basis $\{\beta_i^\cF\}_i$ of $E_{*}^A(\bC P(\cU_A)_+)$, and we write $b_i^- \coloneqq \Th^-(\beta_i^\cF)$, leaving the flag implicit.
	The above reduces to
	\begin{align*}
		E_*^A(\MUP_A)&\cong E_*^A[b_0^-, b_1^-, b_2^-, \cdots][(u^-_\alpha)^{-1}\mid \alpha\in A^*],\\
		E_*^A(\mUP_A)&\cong E_*^A[b_0^-, b_1^-, b_2^-, \cdots][(u^-_\eps)^{-1}].
	\end{align*}
	For $x\in E_*^A(\MUP_A)$, we write $\widehat x$ for $(u^-_{V_1})^{-1}\cdot x$.
	Note that $\beta_0^\cF = \vartheta_{V_1}$.
	There is an equivalence of rings
	\[
		E_*^A(\MU_A)\cong E_*^A[\widehat {b_1^-}, \widehat {b_2^-}, \widehat {b_3^-},\cdots ][(\widehat {u_{\alpha}^-})^{-1}\mid \alpha\in A^*].
	\]
	If $V_1 = \eps$, there is an equivalence
	\[
		E_*^A(\mU_A)\cong E_*^A[\widehat {b_1^-}, \widehat {b_2^-}, \widehat {b_3^-}, \cdots].
	\]
\end{corollary}
\begin{proof}
	Analogous to \cref{lemma:periodicity} and \cref{corollary:A-coordinatized}.
\end{proof}
We also obtain an intermediate localization result to pass from geometric to homotopical bordism.
In the non-periodic case, we need to shift the classes $u_{\alpha}^-$ to lie in $E_*^A(\mU_A)$.
\begin{corollary}
	Writing $\widehat{u_{\alpha}^-}$ for $(u_{\eps}^-)^{-1}\cdot u_{\alpha}^-$, the inclusions $\MU_A\to \MUP_A$ and $\mU_A\to \mUP_A$ induce equivalences
	\begin{align*}
	E_*^A(\mUP_A)[(u_{\alpha}^-)^{-1}\mid \alpha\in A^*\backslash\{\eps\}]\xrightarrow{\cong} E_*^A(\MUP_A),\\
	E_*^A(\mU_A)[(\widehat{u_{\alpha}^-})^{-1}\mid \alpha\in A^*\backslash\{\eps\}]\xrightarrow{\cong} E_*^A(\MU_A).
	\end{align*}
\end{corollary}

\section{The Universality of $\MU_A$}\label{section:6}

The spectrum $\MU$ is the initial oriented (non-equivariant) spectrum, and it gives rise to the universal formal group law over the ring $\pi_*(\MU)$.
The notion of an \emph{equivariant formal group law} was given in \cite{CGK00}.
Hausmann \cite{H22} proved that the ring $\pi_*^A(\MU_A)$ is the representing ring for this structure, the \emph{equivariant Lazard ring}.
Prior to this, in \cite{CGK02} other universal properties of $\MU_A$ were established, in part based on the calculation of the homology of $\MU_A$.
Perhaps the most central is the fact that it is the initial oriented $A$-spectrum.
As the homology of $\MU_A$ is not as it is claimed in the above, these results require our attention, and further care has to be taken.

\begin{construction}
	Recall the functor $\Omega^+\colon \bN\to \Sp_A$ defined for a flag $\cF\colon \bN\to \cU_A$, given on objects by  $\Omega^+(n) = \Omega^{V_n}\MRep_A$.
	We define a restriction of this functor from the following subspectra:
	\[
		\Omega^+_0\colon \bN\to \Sp_A, \quad n\mapsto \Omega^{V_n}\M(n)_A.
	\]
	There is a cocone on $\Omega^+$ given by maps
	\[
	\Psi_V^+\colon \Omega^V\MRep_A \xrightarrow{\Omega^V \incl}\Omega^V\MUP_A \xleftarrow[\simeq]{(\sigma_{V_n}\cdot -)} \MUP_A,
	\]
	which induces an equivalence $\colim_{n\in \bN}\Omega^{V_n}\MRep_A\xrightarrow{\simeq} \MUP_A$, see \cref{proposition:MRep-colimit}.
	The cocone maps restrict to $\Omega^+_0$ and yield maps $\Omega^{V_n}\M(n)_A\to \MU_A$.
	As we always restrict to summands, which are preserved by the structure maps, we obtain an equivalence
	\[
		\colim_{n\in\bN} \Omega^{V_n}\M(n)_A\xrightarrow{\simeq} \MU_A.
	\]
	The map of $\cU_A$-prespectra $i_1\colon \M(1)^{\cU_G}\to \Sigma^{\infty,\cU_A}\bC P(\cU_A)^{\gamma_1}$ is given in level $V$ by the inclusion map $S^V\sm \bC P(V)\hookrightarrow S^V\sm \bC P(\cU_A)$.
	It is a $\underline{\pi}_*$-isomorphism as noted in the proof of \cref{proposition:MRep-htp-grps}.
	Using it, we define the following composite:
	\[
		t^{\uni}\colon \Omega^\eps\Sigma^\infty l(\bC P(\cU_A)^{\gamma_1}) \xleftarrow{\simeq} \Omega^{\eps}\M(1)_A\xrightarrow{\Psi^{+}_\eps} \MU_A.
	\]
\end{construction}

\begin{definition}
	Let $t^\uni\in \MU_A^2(\bC P(\cU_A)^{\gamma_1})$ be the cohomology class associated to the above composite.
	For an irreducible representation $\alpha$ we also define $\chi^\uni(\alpha) \coloneqq \Th(c_\alpha)^{*}(t^\uni)\in \MU_A^2(S^\alpha)$.
\end{definition}

\begin{lemma}\label{lemma:univ-orientation}
	Restricting along the zero section for the bundle $\gamma_1$,  the class
	\[
		x^{\uni}(\eps)\coloneqq s_0^*(t^{\uni})\in \MU_A^2(\bC P(\cU_A),\bC P(\eps))
	\]
	defines an orientation for $\MU_A$, the \emph{universal orientation}.
	The classes $t^\uni$ and $\chi^\uni(\alpha)$ are the associated Thom classes to this orientation.
\end{lemma}
\begin{proof}
	By \cref{lemma:thom-class-is-orientation}, this is equivalent to all $\chi^\uni(\alpha)$ being $\RO(A)$-graded units, and that $\chi^\uni(\eps) = \Sigma^2 1$. 
	The class $\chi^{\uni}(\alpha)$ is equivalently represented by the composite $\Omega^{\eps}\bS^\alpha\to \Omega^{\eps}\M(1)_A \to \MU_A$.
	In $\Sp_A$, tensoring with representation spheres is invertible.
	Hence the above map is found by restricting the following composite to the subspectrum $\MU_A\otimes \bS^\eps$ of the codomain:
	\[
		\bS^\alpha\xrightarrow{\sigma_\alpha}\M(1)_A\hookrightarrow  \MRep_A\hookrightarrow \MUP_A\xrightarrow{(\tau_\eps\cdot -)}\MUP_A\otimes \bS^\eps.
	\]
	In the case of $\alpha = \eps$, we apply \cref{lemma:thom-classes-yoga} and find that the composite above agrees with $ \tau_\eps \cdot \sigma_\eps = \bS^\eps\otimes \iota$, where $\iota$ denotes the unit of $\MUP_A$.
	By the same result, we see that for general $\alpha$, the above composite still defines an $\RO(G)$-graded unit.
	As stated before, the composite already lands in $\MU_A\sm S^\eps$, and so we conclude.
\end{proof}

The following may serve to elucidate the naming of the unshifted Thom classes:
\begin{lemma}\label{lemma:unshifted-thom-classes}
	For a representation $V$, it holds that $\chi^{\uni}(V) = \sigma_V\cdot \tau_{|V|\cdot \eps}\in \MU_A^{2|V|}(S^V)$.
\end{lemma}
\begin{proof}
	By multiplicativity, we reduce to the case of an irreducible $\alpha = V$.
	The class $\chi^{\uni}(\alpha)$ is represented by a map 
	\[
	\Omega^\eps \bS^\alpha \xrightarrow{\Omega^\eps(\sigma_\alpha) } \Omega^\eps \M(1)_A\xrightarrow{\Psi^+_\eps}\MU_A,
	\]
	but $\Psi_\eps^+$ is given by an inclusion and the inverse of $(\sigma_\eps\cdot -)$, i.e. the multiplication with $\tau_\eps$, c.f. \cref{lemma:thom-classes-yoga}.
\end{proof}

We restate \cref{proposition:intro-lim1} and give a proof.
\begin{proposition}\label{proposition:lim1-vanishing}
Let $E$ be an orientable $A$-ring spectrum.
The maps $\Phi_V^+\colon \Omega^V\M(n)_A\to \MU_A$ assemble into a cocone on $\Omega^+_0$ leading to an isomorphism
	\[
		\lim_{n\in \bN^{\op}} E^*_A(\Omega^{V_n}\M(n))\xlongrightarrow{\simeq} E^*_A(\MU_A).
	\]
\end{proposition}

\begin{proof}
	We must calculate the $\limone$-term for the Milnor short exact sequence
	\[
		0\to \limone_{n\in \bN^{\op}}{E^*_A}(\Omega^{V_n}\M(n)^{V_n})\to \lim_{n\in \bN^{\op}}{E^*_A}(\Omega^{V_n}\M(n)^{V_n})\to {E^*_A}(\MU_A)\to 0.
	\]
	It will suffice to show that all transition maps are surjective in cohomology. 
	To this end, consider the map $i\colon V_n\to V_{n+1}$ of codimension $\alpha = \alpha_{n+1}$.
	The induced transition map is given by the map 
		${E^*_A}(\Omega^{V_{n+1}}\M(n+1))\to {E^*_A}(\Omega^{V_n}\M(n))$,
	which agrees up to suspensions with the effect of the structure map $E^*_A(\Th(\alpha\oplus -)) \colon {E^*_A}(\M(n+1))\to {E^*_A}(S^\alpha\otimes \M(n))$.
	Applying the Thom isomorphism, this map corresponds to 
	\[
		E_*^A(\alpha\oplus -)\colon E_*^A(\Gr_{n+1}(\cU_A)_+\to E_*^A(\Gr_{n+1}(\cU_A)_+).
	\]
	A section of this map would yield surjectivity, which is equivalent to a retraction for the $E_*^A$-dual map, the map on homology: 
	\[
		\vartheta_\alpha\otimes -\colon E_*^A(\Gr_{n}(\cU_A)_+)\to E_*^A(\Gr_{n+1}(\cU_A)_+).
	\]
	In the case that $\alpha = V_0$, we find that $\vartheta_\alpha = \beta^\cF_0$, and a retraction is given by the linear extension of
	\[
		\beta_{i_1}^\cF\otimes \cdots \otimes \beta_{i_{n+1}}^\cF \mapsto 
		\begin{cases}
			\beta^\cF_{i_2}\otimes \cdots \otimes \beta^\cF_{i_{n+1}} & \text{if } i_1 = 0,\\
			0 & \text{else.}
		\end{cases} 
	\]
	Here we use the convention that the indices in the monomials are weakly increasing.
	For a general $\alpha$, we may simply pick a flag including $\alpha$ as its first term.
\end{proof}

\begin{definition}
	Let $\CRing_A$ be the category of commutative $A$-ring spectra and homotopy classes of ring maps.
	For $E\in \CRing_A$, define $\Or(E)$ as the set of orientations $x(\eps)\in E^2_A(\bC P(\cU_A), \bC P(\eps))$.
	This assembles into a functor $\Or\colon \CRing_A\to \Set$ by pushing forward the orientation, see \cref{lemma:pushforward-orientation}. 
\end{definition}

We thus recover \cite[Thm. 1.2.]{CGK02}, or \cref{theorem:intro-universality} from the introduction.
Let us recall the statement and give a proof.
\begin{theorem}\label{theorem:classifies-orientations}
	Pushing forward the tautological orientation of $\MU_A$ defines a natural isomorphism of functors 
	\[
		\Hom_{\CRing_A}(\MU_A,E)\to  \Or(E), \quad f\mapsto f_* (t^{\MU_A})
	\]
	from $\CRing_A$ to $\Set$. In other words, $\MU_A$ corepresents orientations among equivariant commutative $A$-ring spectra.
\end{theorem}

\begin{proof}
	We show surjectivity first.
	Let $x(\eps)$ be an orientation of $E$, and let $t$ be the associated Thom class, given by a morphism $t\colon \Omega^\eps\M(1)_A\to E$.
	As discussed in \cref{theorem:orientation-yields-thom-classes}, one obtains universal Thom classes $t_n\colon \Omega^{\eps^n}\M(n)_A\to E$ such that the multiplication map $\mu\colon \M(n)_A\otimes \M(m)_A\to \M(n+m)_A$ pulls back $t_{n+m}$ to $t_n\otimes t_m$.
	Further, we obtain Thom classes $\chi(V)\colon \Omega^{|V|\cdot \eps}S^V\to E$.
	These are $\RO(A)$-graded units whose inverses we denote by $\overline{\chi}(V)\colon S^{|V|\cdot \eps}\to E\otimes S^V$.
	The $\chi(V)$ and hence $\overline{\chi}(V)$ are multiplicative for direct sums in $V$.
	For a flag $\cF\colon \bN\to s(\cU_A)$, the transition maps of the diagram $\Omega^+_0$ are given by
	\[
		\Omega^{V_n}\M(n)_A\xrightarrow{\id\sm \sigma_{\alpha_{n+1}}}\Omega^{V_n}\M(n)_A\otimes \Omega^{\alpha_{n+1}}\M(1)_A\to \Omega^{V_{n+1}}(\M(n)_A\otimes \M(1)_A) \xrightarrow{\Omega^{V_{n+1}}\mu} \Omega^{V_{n+1}}\M(n+1),
	\]
	where $\mu$ is induced by the multiplication of $\MRep_A$.
		We extend this diagram to a cocone over $E$ by utilizing the maps $\overline{\chi}(V_{n})\cdot t_n\colon \Omega^{V_n}\M(n)_A\to E$.
	These maps will define a cocone precisely if certain diagrams commute, the first of which is pictured below.
	Letting $\alpha = V_1$, we show that the triangle 
	\[
		\begin{tikzcd}
		\bS && {\Omega^{\alpha}M(1)} \\
		& E
		\arrow["{\overline{\tau}_{\alpha}}", from=1-1, to=1-3]
		\arrow["{1_E}"', from=1-1, to=2-2]
		\arrow["{\overline{\chi}(\alpha)\cdot t_1}", from=1-3, to=2-2]
		\end{tikzcd}
	\]
	commutes.
	As tensoring with $\bS^\eps$ or $\bS^\alpha$ is invertible, it suffices to show that the following composite is equivalent to $1_E\otimes\bS^\alpha \otimes \bS^\eps$:
	\[
	\bS^\alpha \otimes \bS^\eps\xrightarrow{\sigma_\alpha\sm \id}   \M(1)_A\otimes \bS^\eps \xrightarrow{t_1\otimes \overline{\chi}(\alpha)}E\otimes \bS^\eps\otimes E\otimes \bS^\alpha \xrightarrow{\mu_E \otimes \mathrm{swap}} E\otimes \bS^\alpha\otimes \bS^\eps.
	\]
	The composite $t_1\circ \sigma_\alpha\colon \bS^\alpha\to E\otimes \bS^\eps$ represents the element $\chi(\alpha)$ up to (de)suspending by representation spheres.
	Thus, the composite here depicted is given by the product of $\chi(\alpha)$, inverse $\overline{\chi}(\alpha)$, and $1_E$.
	The commutativity of the other diagrams making up the cocone follows by multiplicativity: the Thom classes $\overline{\chi}(V)$ are multiplicative and $\mu$ pulls back $t_{n+1}$ to $t_n\otimes t_1$.
	The cocone then provides a map from the colimit $T\colon \colim_{n\in \bN} \Omega^{V_n}\M(n)_A \simeq \MU_A\to E$, such that $T_*(x^\uni(\eps)) = x(\eps)$.
	It remains to show that $T$ is a ring map, or equivalently that the two maps from $\MU_A\otimes \MU_A$ to $E$ given by $\mu^E\circ T\otimes T$ and $T\circ \mu^{\MU_A}$ agree.
	We have 
	\[
		\MU_A\otimes \MU_A\simeq \colim_{n\in \bN}(\Omega^{V_n} \M(n)_A)\otimes(\Omega^{V_n}\M(n)_A).
	\]
	and in $E_A^*$-cohomology, the $\limone$-term vanishes once again, by the freeness of the homology of $\Omega^{V_n} \M(n)_A$ and \cref{proposition:lim1-vanishing}.	
	Hence suffices to show that both maps yield equivalent cocones on this diagram.
	The multiplicativity of Thom classes $\chi(V)$ and $t_n$ tells us that the cocones indeed agree.
	Next is the question of injectivity.
	Consider a pair of ring maps $g,h\colon \MU_A\to E$, which induce the same orientation $x(\eps) = g_*(x^{\MU_A}(\eps)) = h_*(x^{\MU_A}(\eps))$.	
	For a flag $\cF\colon \bN\to s(\cU_G)$, one has $\MU_A\simeq \colim_{n\in \bN} \Omega^{V_n}\M(n)_A$ along transition maps detailed e.g. at the start of this proof.
	We will show that the restriction of $g$ and $h$ to each $ \Omega^{V_n}\M(n)$ agree.
	Firstly, note that the composite $\Omega^{V_n}\M(n)\to \MU_A$ is given by $\overline{\chi}^{\MU_A}(V_{n})\cdot t_n^{\MU_A}$.
	As $g$ and $h$ agree on the orientation class and are ring maps, we find that they agree on this product of (inverse) Thom classes.
	By the vanishing of the $\limone$-term, we conclude that $g$ and $h$ already agree on $\MU_A$.
\end{proof}

\begin{remark}
	At this stage, it we point out that the flatness of $E_*^A(\MU_A)$ suffices to recover the subsequent results  \cite[8.3.;8.4.;8.5.]{CGK02}.
\end{remark}


\end{document}